\documentclass{article}

\usepackage{bbm}
\usepackage{amsthm}
\usepackage[utf8]{inputenc}
\usepackage[english]{babel}
\usepackage{color}
\usepackage{amsmath,amssymb,amsthm,yhmath,amsfonts,mathrsfs,enumitem,mathtools,dsfont}
\usepackage{tikz}
\usetikzlibrary{shapes,arrows,positioning}

\newcommand{\ul}{\underline}
\newcommand{\pr}{\mathbb{P}}
\newcommand{\E}{\mathbb{E}}

\renewcommand{\d}{\mathrm{d}}

\newcommand{\1}[1]{{\mathds{1}}_{\left\{#1\right\}}}

\newcommand{\K}{\mathcal{K}}
\newcommand{\N}{\mathcal{N}}
\newcommand{\HH}{\mathcal{H}}
\renewcommand{\S}{\mathcal{S}}

\newtheorem{theorem}{Theorem}

\newtheorem{remark}{Remark}
\newtheorem{lemma}{Lemma}
\newtheorem*{lemma*}{Lemma}
\newtheorem{prop}{Proposition}

\DeclareMathOperator{\cov}{cov}

\newcommand{\rem}[1]{{#1}}

\newcommand{\remm}[1]{{#1}}

\title{Accuracy criterion for mean field approximations of Markov processes on hypergraphs
	\footnote{Supported by the \'UNKP-21-1 New National Excellence Program of the Ministry for Innovation and Technology from the source of the National Research, Development and Innovation Fund. 
		\includegraphics[width=0.05\textwidth]{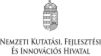}
		\includegraphics[width=0.05\textwidth]{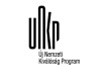}
	} 
	\footnote{ Partially supported by the ERC Synergy under Grant No. 810115 - DYNASNET.}
}

\author{
	D\'aniel Keliger\\
	{\small Department of Stochastics, Institute of Mathematics,}\\
	{\small Budapest University of Technology and Economics}\\
	{\small M\H{u}egyetem rkp. 3., H-1111 Budapest, Hungary;
	}\\
	{\small Alfr\'ed R\'enyi Institute of Mathematics, Budapest, Hungary}\\
	{\small e-mail: perfectumfluidum@gmail.com} \\[3mm]
	Ill\'es Horv\'ath\\
	{\small MTA-BME Information Systems Research Group}\\
	{\small e-mail: horvath.illes.antal@gmail.com}\\[3mm]
}

\begin{document}
	\maketitle
	
	\begin{abstract}
		We provide error bounds for the $N$-intertwined mean-field approximation (NIMFA) for local density-dependent Markov population processes with a well-distributed underlying network structure showing NIMFA being accurate when a typical vertex has many neighbors. The result justifies some of the most common approximations used in epidemiology, statistical physics and opinion dynamics literature under certain conditions. We allow interactions between more than 2 individuals, and an underlying hypergraph structure accordingly.
	\end{abstract}
	
	\section{Introduction}
	
	The analysis of stochastic population processes is an important topic in several disciplines, such as epidemiology, biology, economics or computer systems \cite{Bruneo2011, Bakhshi2010a,hht2014,Caravagna2011,Schlicht2008}. Such processes consist of a large number of interacting individuals (agents) that execute random actions based on the behavior of other individuals.
	
	A widely-used framework is Markov population processes, where each individual is in a local state from a fixed, finite state space, and can change their state in a Markovian manner. For such models, the state space increases exponentially with the population size, making an exact analysis infeasible even for moderate population sizes, instead raising the question of good approximations as the next best thing.
	
	The classical result of Kurtz \cite{kurtz70,kurtz78} is based on two main assumptions: that each individual can observe the entire population, and that the Markovian transition rates of each individual depend on the observation in a density-dependent manner. The conclusion is that, as the number of individuals diverges, the evolution of the stochastic system converges to a deterministic mean-field limit. This limit is straightforward to compute numerically, and can serve as a good approximation of the stochastic system when the number of individuals is large. The mean-field limit of Kurtz is referred to as the \emph{homogeneous mean-field approximation} in the present paper.
	
	While the density-dependent Markov setting is flexible and covers many potential applications, the assumption that each individual can observe the entire population is very restrictive. In many population processes arising from real-life examples, individuals do not have full information about the entire population; instead, each individual can observe only a subset of the population. This information structure can be described by a network topology, where each individual has interactions only with its neighbors according to that topology.
	
	The $N$-intertwined mean field approximation (NIMFA) \cite{NIMFA2011} is a quenched mean-field approximation, where differential equations are considered for each individual based on their expected evolution. NIMFA is a deterministic process different from the homogeneous mean-field approximation that incorporates the network structure naturally, making it a potentially more accurate approximation. On the flip side, the computational complexity is considerably increased compared to the homogeneous mean-field approximation; nevertheless, it remains tractable for population sizes large enough to make it relevant for practical applications. Unfortunately, unlike for homogeneous systems, the justification for using NIMFA is poorly understood, mostly relying on numerical evidence \cite{NIMFA_accuracy,NIMFA_accuracy2} along with a few theoretical results \cite{simon2017NIMFA,Sridhar_Kar1, Sridhar_Kar2,age_structured}.
	
	In the present paper, we focus on a specific class of Markov processes dubbed \emph{local density-dependent Markov population processes}, which preserves the density-dependent assumption of Kurtz, but allows an underlying network structure that dictates the environments observed by each individual. This setting covers many of the frequently used stochastic models, such as the SIS process in epidemiology \cite{simplicial1, simplicial2, simplicial3, simplicial4}, Glauber dynamics in statistical physics \cite{Glauber_original,Glauber_ODE}, or the voter model and majority vote in opinion dynamics \cite{voter_majority,voter_adaptive}.   We incorporate interactions between more than 2 vertices into the model with an underlying hypergraph structure accordingly to reflect on some recent developments in the theory of higher order interactions.  
	
	We provide general error bounds for NIMFA that are strong on well-distributed networks. Furthermore, under additional homogeneity assumptions, such as annealed or activity driven networks \cite{annealed, activity1} we show these error bounds to be small, with the added benefit of further reducing the number of equations to other well-known approximations, like the \emph{heterogenous mean field approximation} \cite{vesp}. Finally, we elaborate the on the argument given by K. Devriendt and P. Van Mieghem \cite{unified_meanfield} and show that Szemer\'edi's regularity lemma \cite{Szemeredi} can be applied to reduce the number of equations (depending on a given $\varepsilon$ error).
	
	The rest of the paper is structured as follows. Section \ref{s:setup} introduces basic notation and 
	setup for
	density-dependent Markov population processes along with examples of models used in the literature to illustrate these concepts and their applicability. Section \ref{s:errorbounds} states the main results and also relates them to the recent work of Sridhar and Kar \cite{Sridhar_Kar1, Sridhar_Kar2} and Parasnis et al. \cite{age_structured}. Section \ref{s:red} discusses further reductions of NIMFA 
	to more simple approximations used throughout the literature. Section \ref{s:Discussion} contains a summary of this paper along with the limitations of these results and possible directions for further research.

	Finally, proofs are contained in Section \ref{s:proofs}.
	
	\section{Setup}
	\label{s:setup}
	
	\subsection{The underlying hypergraph}
	
	Let $G$ be a finite hypergraph on $N$ vertices. The vertex set is labeled $[N]=\{1, \dots, N\}$. The hypergraph is not necessarily uniform; edges may contain up to $M+1$ vertices. The edges are ordered, with the first vertex being special, and we will usually use the notation $(i,j_1,\dots,j_{m})$ for an edge where $1\leq m\leq M$ and $i,j_1,\dots,j_{m}\in [N]$. The idea behind the distinction of the first vertex in an edge is that $w^{(m)}_{i,j_1,\dots,j_{m}}$ will describe the strength of connections where $j_1,\dots,j_{m}$ have a joint effect on vertex $i$ (see Figure \ref{fig_hyperedge}).
	
	\begin{figure}[t]
		\centerline{\includegraphics[width=0.3\textwidth]{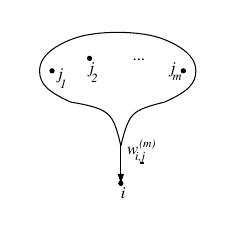}}
		\caption{Edge (hyperedge) with weight $w^{(m)}_{i,j_1,\dots,j_{m}}$.}
		\label{fig_hyperedge}
	\end{figure}

	The $M=1$ case corresponds to (directed) graphs.
	
	We allow so-called \emph{secondary loops} (abbreviated as s. loop), which are $(i,\ul{j})$ edges with non-distinct vertices among $j_1,\dots,j_{m}\in [N]$. Note that traditional loops for the $m=1$ case are excluded from this definition.
	
	We use the notation $[N]^m$ to denote the set of $m$-tuples, and $\ul{j}$ abbreviates $(j_1,\dots,j_{m})$.
	
	For unweighted hypergraphs, adjacency indicators $a^{(m)}_{i,j_1,\dots,j_{m}}$ (where $1\leq m\leq M$ and $i,j_1,\dots,j_{m}\in [N]$)
	
	\begin{align*}
		a_{i,\ul{j}}^{(m)} = \begin{cases}
			1 \ \text{ if $i,j_1,\dots,j_m$ are on the same hyperedge} \\
			0 \ \text{else}
		\end{cases}
	\end{align*}
	
	describe the connections between the vertices.
	
	
	In-degrees for $1\leq m\leq M$ are defined as
	\begin{align}
		\label{eq:dmi}
		d^{(m)}(i):=&\frac{1}{m!}\sum_{\underline{j} \in [N]^m} a_{i,\underline{j}}^{(m)},
	\end{align}
	(where $m!$ is included to cancel the re-orderings of $\ul{j}$), and the average in-degree for each $1\leq m\leq M$ is
	\begin{align*}
		\bar{d}^{(m)}:=& \frac{1}{N} \sum_{i=1}^{N}d^{(m)}(i).
	\end{align*}
	
	In the literature, some normalization is usually assumed. In the present paper, we introduce normalized weights $w^{(m)}_{i,j_1,\dots,j_{m}}$ and corresponding normalized in-degrees 
	$$\delta^{(m)}(i):=\sum_{\underline{j} \in [N]^m} w_{i,\underline{j}}^{(m)}.$$
	representing the total weight of $m$-interactions effecting vertex $i\in [N]$. 
	In the $M=1$ case (classical graphs) we tend to omit the upper index $(m)$ and write $w_{i,\ul{j}}^{(m)}$ simply as $w_{ij},$ and we also utilize the matrix notation $W=(w_{ij})_{i,j \in [N]}.$ 
	
	We have two Conventions for the normalization.
	
	Convention 1:
	\begin{align}
		\label{eq:conv1}
		w_{i,\underline{j}}^{(m)}=\frac{a_{i,\underline{j}}^{(m)}}{m! \bar{d}^{(m)}},\qquad \delta^{(m)}(i)=\frac{d^{(m)}(i)}{\bar{d}^{(m)}}.
	\end{align}
	
	Convention 2:
	\begin{align}
		\label{eq:conv2}
		w_{i,\underline{j}}^{(m)}=\frac{a_{i,\underline{j}}^{(m)}}{m! d^{(m)}(i)}, \qquad
		\delta^{(m)}(i)=1.
	\end{align}
	
	(The same $m!$ from \eqref{eq:dmi} is now included in the conventions.)
	
	For either convention, whenever the denominator would be 0, the numerator will also be 0, and $w_{i,\underline{j}}^{(m)}$ is simply set to 0 as well.
	
	We set
	\begin{align*}w_{\max}=\max_{i,\ul{j},m}w_{i,\underline{j}}^{(m)}.
	\end{align*}
	
	\remm{
	Furthermore, we define
	\begin{align*}
		w_{\max}^{*}:=\max_{m,i} \sum_{ \substack{\ul{j},\ul{k} \in [N]^m \\ \ul{j} \cap \ul{k} \neq \emptyset }}w_{i,\ul{j}}^{(m)}w_{i,\ul{k}}^{(m)}.
	\end{align*}
	
	Note that in the $M=1$ case we have
	\begin{align*}
		w_{\max}^*= \max_i \sum_{j}w_{ij}^2 \leq \delta_{\max} w_{\max}.
	\end{align*}
	}

	We are going to set regularity assumptions for the weights and degrees:
	\begin{align}
		\label{A:1}
		\delta^{(m)}(i) \leq&\, \delta_{\max}, \\
		\label{A:2}
		\sum_{\substack{\underline{j} \in [N]^m \\ \underline{j} \textrm{ s. loop} }} w_{i,\underline{j}}^{(m)} \leq&\, R \sqrt{\remm{w_{\max}^{*}}}.
	\end{align}
	
	For Convention 2, \eqref{A:1} always holds. For Convention 1, we need $d^{(m)}(i) \leq \delta_{\max} \bar{d}^{(m)}$ (upper regularity of the hypergraph).
	
	\eqref{A:2} always holds for $M=1$. It also obviously holds if there are no secondary loops. In other cases, it is an actual restriction on the total weight of secondary loops.
	
	Symmetry is in general not assumed, that is, the hypergraph may be directed.

	For some results concerning classical graphs ($M=1$) with Convention 2, the extra assumption is needed for \emph{out-}degrees as well.
	
	\begin{align}
		\label{eq:out_degree}
		\delta^{\textrm{out}}(j):=\sum_{i \in  [N]} w_{ij} \leq \delta_{\textrm{max}}^{\textrm{out}}
	\end{align}
	 
 	Assumption \eqref{A:1} and \eqref{eq:out_degree} can be understood as a weaker version of double stochasticity of $W$ assumed in \cite{Sridhar_Kar1, Sridhar_Kar2}.

	
	
	
	\subsection{Local density dependent Markov population process}
	\label{s:Markov}
	
	We define a Markov process on the hypergraph. Each vertex is in a state from a finite state space $\S$. $\xi_{i,s}(t)$ denotes the indicator that vertex $i$ is in state $s$ at time $t$; the corresponding vector notation is
	$$\xi_i(t)=\left(\xi_{i,s}(t)\right)_{s \in \S}.$$
	
	We also introduce the notation
	$$\xi^{(m)}_{\ul{i},\ul{s}}(t)=\prod_{k=1}^m \xi_{i_k,s_k}(t),$$
	where $\ul{i}=(i_1,\dots,i_{m})$ is an edge and $\ul{s}=(s_1,\dots,s_{m})$ is a collection of states ($s_k\in \S,k=1,\dots,m)$. $\xi^{(m)}_{\ul{i},\ul{s}}(t)$ describes the indicator of vertices $i_1,\dots, i_m$ being in states $s_1,\dots,s_n$ at time $t$, respectively.
	
	We define the \emph{$m$-neighborhood} of vertex $i$ corresponding to $\ul{s}=(s_1,\dots,s_{m})$ as
	\begin{align}
		\label{eq:phidef}
		\phi^{(m)}_{i,\ul{s}}(t) = \sum_{\ul{j}\in [N]^m} w^{(m)}_{i,\ul{j}}\xi^{(m)}_{\ul{j},\ul{s}}(t).
	\end{align}
	Some explanation is in order. Let $\ul{s}=(s_1,\dots,s_{m})$ be fixed for now. According to \eqref{eq:phidef}, we consider all edges that include $i$ and $m$ other vertices; for each such edge, we check whether the $m$ other vertices are exactly according to the configuration of states described by $\ul{s}$; if yes, their contribution to $\phi^{(m)}_{i,\ul{s}}(t)$ is $w^{(m)}_{i,\ul{j}}$, otherwise their contribution is 0.
	
	The $m$-neighborhoods of $i$ consist of $\phi^{(m)}_{i,\ul{s}}(t)$ for all possible configurations of states $\ul{s}$. The corresponding vector notation is
	\begin{align}
		\label{eq:phidef2}
		\phi^{(m)}_{i}(t) = \left(\phi^{(m)}_{i,\ul{s}}(t)\right)_{\ul{s}\in \S^m},
	\end{align}
	and we may even write
	\begin{align}
		\label{eq:phidef3}
		\phi_{i}(t) = \left(\phi^{(m)}_{i}(t)\right)_{m=1}^M
	\end{align}
	for the entire neighborhood of $i$.
	
	In \eqref{eq:phidef}, the normalized weights $w^{(m)}_{i,\ul{j}}$ are used; in case $w^{(m)}_{i,\ul{j}}=0$ for some $\ul{j}$, the corresponding interaction is simply not present.
	

	Each vertex may transition to another state in continuous time. The transition rates of a vertex may depend on all of its $m$-neighborhoods for $1\leq m \leq M$; accordingly, the transition rate from $s'$ to $s$ is described by the function
	$$q_{ss'}: \otimes_{m=1}^M\mathbb{R}^{\S^m} \to \mathbb{R}$$
	for each $s'\neq s\in\S$.
	
	We assume $q_{ss'}$ is locally Lipschitz, and we also require $q_{ss'}(\phi^{(1)},\dots,\phi^{(M)}) \geq 0$ for non-negative inputs.
	
	For ``diagonal'' rates,
	$$q_{ss}:=-\sum_{s' \neq s}q_{s's}$$
	corresponds to the total outgoing rate from state $s$.
	
	The corresponding transition matrix is $Q=\left(q_{ss'}\right)_{s,s' \in \S}.$ We emphasize that in this convention $q_{ss'}$ refers to an $s \leftarrow s'$ transition and not an $s \to s'$ one. This ordering allows us to use column vectors and matrix multiplication from the left. 
	
	The dynamics of $\left(\xi_{i}(t)\right)_{i=1}^{N}$ is a continuous-time Markov chain with state-space $\S^{N}$ where each vertex performs transitions according to the transition rates $q_{s's}$, independently from the others. After a transition, vertices update their neighborhood vectors $\phi_i(t)$. We call such dynamics \emph{local-density dependent Markov processes}.
	
	We define the process $\left(\xi_{i,s}\right)_{i,s}$ formally via Poisson representation:
	\begin{align}
		\label{eq:xi}
		\begin{split}
			{\xi}_{i,s}(t)=&\xi_{i,s}(0)+\sum_{\substack{s' \in \S \\ s' \neq s}} \N_{i,ss'} \left(\HH_{i,ss'}(t) \right)-\N_{i,s's} \left(\HH_{i,s's}(t) \right),\\
			\HH_{i,ss'}(t)=& \left\{(\tau,x) \in \mathbb{R}^2 \left. \right| 0 \leq \tau \leq t, \ 0 \leq x \leq q_{ss'} \left(\phi_{i}(\tau) \right)  {\xi}_{i,s'}(\tau)\right\},
		\end{split}
	\end{align}
	where for each choice of $1\leq i\leq N$ and $s\neq s'\in\S$, $\left(\N_{i,ss'}(x,y):x,y\geq 0\right)$ is a 2-dimensional Poisson-process with density 1, and the processes are independent for different $(i,s,s')$ triples.
	
	\eqref{eq:xi} is a cumulative formula counting all transitions of the vertex $i$ to and from state $s$ up to time $t$; $s \leftarrow s'$ transitions are generated using the Poisson points in the 2-dimensional domain $\HH_{i,ss'}(t)$ which has area $\int_0^t  q_{ss'} \left(\phi_{i}(\tau) \right)  {\xi}_{i,s'}(\tau) \d\tau$, ensuring the proper transition rate for $s \leftarrow s'$ jumps at time $\tau$. The second term of the sum corresponds to $s' \leftarrow s$ transitions in a similar manner.

	\subsection{N-intertwined mean field approximation}
	
	Although the state occupation probabilities of the population process can be described by the Chapman--Kolmogorov equations, the number of equations is $\left |\S \right|^N$, making it infeasible for numeric or analytic investigations even for moderate sized populations. To address this issue, several approximation schemes had been introduced in the literature with varying complexity.
	
	This chapter discusses the quenched mean field approximation \cite{NIMFA2011}, also called the N-intertwined mean field approximation (NIMFA). NIMFA preserves all information regarding the graph structure and only neglects dynamical correlation between vertices. The goal is to derive state occupation probabilities for each vertex separately, resulting in a total of $\left | \S \right | N$ equations. 
	
	A possible intuition for NIMFA is as follows.
	\begin{align}
		\label{eq:preNIMFA}
		\frac{\d}{\d t} \E \left(\xi_{i}(t) \right)= \E \left[Q \left(\phi_{i}(t)\right) \xi_{i}(t) \right]
	\end{align}
	can be derived from \eqref{eq:xi}. To close \eqref{eq:preNIMFA}, we apply the approximation $\phi_{i}(t) \approx \E \left( \phi_i(t) \right)$, which is reasonable when $N$ is large and there is low correlation between vertices:
	\begin{align*}
		\E \left[Q \left(\phi_{i}(t)\right) \xi_{i}(t) \right] \approx  \E \left[Q \left(\E \left(\phi_{i}(t) \right )\right) \xi_{i}(t) \right]=Q \left(\E \left(\phi_{i}(t) \right )\right) \E \left(\xi_{i}(t) \right).
	\end{align*}
	Accordingly, the NIMFA approximation $z_i(t)=(z_{i,s}(t))_{s\in \S},1\leq i\leq N$ is the solution of the system
	\begin{align}
		\label{eq:NIMFA}
		\begin{split}
			\frac{\d}{\d t} z_{i}(t)=& Q \left(\zeta_{i}(t) \right)z_{i}(t), \\
			\zeta_{i}(t)=& \left(\zeta_{i}^{(m)}(t) \right)_{m=1}^M, \\
			\zeta_{i}^{(m)}(t)=& \left(\zeta_{i,\underline{s}}^{(m)}(t)\right)_{\underline{s} \in [N]^m}=\left(\sum_{\underline{j} \in \S^m} w_{i,\underline{j}}^{(m)} z_{\underline{j},\underline{s}}^{(m)}(t) \right)_{\underline{s} \in \S^m},
		\end{split}	
	\end{align}
	where $z_i(t)$ corresponds to $\xi_i(t)$ and $\zeta_i(t)$ corresponds to $\phi_i(t)$, and then the approximation used is
	$$\pr \left(\xi_{i,s}(t)=1\right)=\E \left(\xi_{i,s}(t)\right) \approx z_{i,s}(t).$$
	
	The following theorem ensures the existence and uniqueness of the solution of \eqref{eq:NIMFA}.
	
	\begin{theorem}
		\label{t:NIMFA_Delta}	
		Let $\Delta^{\S}$ denote the set of probability vectors from $\mathbb{R}^{\S}.$ 
		For any initial condition $ z_i(0) \in \Delta^{\S}$ for all $i$ the ODE system \eqref{eq:NIMFA} has a unique global solution such that $ z_{i}(t) \in \Delta^S$ for all $i$ and $t>0$ as well.
	\end{theorem}
	
	\subsection{Examples}
	
	In this section we give some examples for models covered by the formalism of Section \ref{s:Markov}.
	
	\subsubsection*{The simplicial SIS model}
	
	We will use the simplicial SIS model, also referred to as the contact process as a running example.
	
	In the $M=1$ case (graphs) the setup is the following: Each vertex can be in one of two states: susceptible ($S$) and infected ($I$), hence the state space is $\S=\{S,I \}.$ Infected vertices become susceptible at a constant rate $\gamma \geq 0$ while susceptible vertices receive the illness with rate proportional to number of its infected neighbhours.
	
	The number of infected neighbhours of vertex $i \in [N]$ at time $t$ equals to
	$$ \sum_{j=1}^N a_{ij}\xi_{j,I}(t) $$
	as $a_{ij}\xi_{j,I}(t)$ the indicator of vertex $j$ is connected to vertex $i$ and that it is infected at time $t$. After normalizing it with $\bar{d}$ or $d(i)$ depending on our choice of convention 1 or 2 one gets
	$$\sum_{j=1}^N w_{ij} \xi_{j,I}(t)=\phi_{i,I}(t). $$
	
	Therefore, the transition rates takes the form $q_{SI}(\phi_i(t))=\gamma, \ q_{IS}(\phi_i(t))=\beta \phi_{i,I}(t)$ where $\beta \geq 0$ is a suitable constant factor. In matrix form:
	\begin{align*}
		Q(\phi_i(t))=
		\left[ {\begin{array}{cc}
				-\gamma & \gamma \\
				\beta \phi_{i,I}(t) & -\beta \phi_{i,I}(t) \\
		\end{array} } \right]
	\end{align*}
	
	For the SIS process NIMFA takes the form:
	\begin{align*}
		\frac{\d}{\d t} z_{i,I}(t)=-\gamma z_{i,I}(t)+\beta(1-z_{i,I}(t))\sum_{j=1}^{N}w_{ij}z_{j,I}(t).
	\end{align*}
	Here we used $z_{i,S}(t)=1-z_{i,I}(t)$ which is also the reason why it enough to write the $I$ components only.
	
	The extension of the SIS model to hypergraphs is called the simplicial SIS model. The curing rate stays $\gamma$, however the infection dynamics is modified. A susceptible vertex can be infected via any $(m+1)$-edge if all other $m$ vertices are infected. The weighted sum of such edges $(m+1)$-edges is
	$$ \sum_{\ul{j} \in [N]^{m}} w_{i,\ul{j}}^{(m)} \xi_{\ul{j},(I,\dots, I)}^{(m)}(t)=\phi_{i,(I,\dots,I)}^{(m)}(t). $$
	
	The infection rates is sum of all the $1 \leq m \leq M$ with appropriate $\beta_1,\dots,\beta_M \geq 0$ factors:
	$$q_{IS}(\phi_i(t))=\sum_{m=1}^{M}\beta_m \phi_{i,(I,\dots,I)}^{(m)}(t). $$
	
	For the simplicial SIS model NIMFA takes the form
	\begin{align*}
		\frac{\d}{\d t} z_{i,I}(t)=-\gamma z_{i,I}(t)+\left(1-z_{i,I}(t) \right)\sum_{m=1}^{M}\beta_m \sum_{\underline{j} \in [N]^m}w_{i,\underline{j}}^{(m)}z_{\underline{j},(I,\dots,I)}^{(m)}(t),
	\end{align*}

	\subsubsection*{Glauber dynamics}
	
	Glauber dynamics is a stochastic process whose stationary distribution coincides with the distribution given by a spin system, such as the Ising model \cite{Glauber_original}.
	
	There are two possible states: $\S=\{+,-\}.$ Instead of the indicators 
	$$\xi_{i,+}(t),\, \xi_{i,-}(t)$$ 
	it is customary use the sign variables
	$$\sigma_i(t):=\xi_{i,+}(t)-\xi_{i,-}(t)=2 \xi_{i,+}(t)-1.$$ 
	
	In physical systems it is natural to assume $w_{ij}$ is symmetric and $w_{ii}=0$.
		
	The dynamics is the following:
	\begin{itemize}
		\item At each time step, choose a vertex $i$ uniformly.
		\item With probability $p_i(\sigma)=\frac{e^{\beta S_i(\sigma)}}{e^{\beta S_i(\sigma)}+1}$, vertex $i$ switches to state + (else -), where
		$$S_i(\sigma)=\sum_{j=1}^Nw_{ij}\sigma_j. $$
	\end{itemize}
	Note that $S_{i}(\sigma)$ arises from the reduction of the energy
	\begin{align*}
		H(\sigma):=-\frac{1}{2}\sum_{i<j}w_{ij}\sigma_{i}\sigma_{j}
	\end{align*}
	when vertex $i$ is flipped from $-$ to $+$.  The stationary distribution is then given by the Gibbs measure
	\begin{align*}
	P(\sigma):=& \frac{1}{Z}e^{-\beta H(\sigma)}, \\
	Z:=& \sum_{\sigma}e^{-\beta H(\sigma)}.
	\end{align*}

	We modify the above dynamics. First, note that, in accordance with \eqref{eq:phidef},
	\begin{align*}
	S_i(\sigma(t))=\sum_{j=1}^{N}w_{ij}\left(\xi_{j,+}(t)-\xi_{j,-}(t) \right)=\phi_{i,+}(t)-\phi_{i,-}(t).
	\end{align*}
	With a slight abuse of notation, we denote
	\begin{align*}
	S\left(\phi_i(t) \right):=\alpha \phi_{i,+}(t)-\gamma \phi_{i,-}(t),
	\end{align*}
	allowing the dynamics to have a preferred state.
	
	Furthermore, we turn to the continuous time version instead with transition rates given by
	\begin{align*}
		q_{+-}(\phi)=&e^{\beta S(\phi)}, \\
		q_{-+}(\phi)=&1.
	\end{align*}
	Since there are only two states, it is enough to consider the probabilities of occupying state +. For this, NIMFA gives the following system of ODEs:
	\begin{align}
	\label{eq:NIMFA_Ising}
	\frac{\d}{\d t} z_{i,+}(t)=&(1-z_{i,+}(t))e^{\beta S(\zeta_i(t))}-z_{i,+}(t).
	\end{align}
	
The equilibrium state is given by the fixed point problem
	\begin{align}
	\label{eq:Ising_stationary}
	 z_{i,+}=\frac{e^{\beta S(\zeta_i)}}{e^{\beta S(\zeta_i)}+1}.
	\end{align}
	Assume $\alpha=1,\gamma=-1$ as in the original setting and that the underlying weighted graph is regular: $\forall i \ \delta(i)=\sum_{j} w_{ij}=1$. Than \eqref{eq:Ising_stationary} reduces to
	\begin{align*}
	&\sigma=\tanh \left(\frac{1}{2}\beta \sigma \right), \\
	&\forall i \ 2 z_{i,+}-1=\sigma
	\end{align*}
	giving back the classical mean field approximation of the Ising model on lattice. This is not surprising as both NIMFA and the classical mean field approach is based on the assumption of independence of vertices.
	
	Based on \cite{Glauber_ODE}, we can generalize the model for hypegraphs via extending $S(\phi)$ to
	\begin{align*}
	S(\phi_i(t)):=\sum_{m=0}^{M} \alpha_m \phi_{i,(+,\dots,+)}^{(m)}(t)-\gamma_m \phi_{i,(-,\dots,-)}^{(m)}(t)
	\end{align*}
	allowing the system to lose even more energy when $3$ or more neighbors have the same configuration on a hyper-edge. 

	\subsubsection*{The voter model}

	The voter model is a conceptually simple stochastic process modeling opinion dynamics \cite{voter_majority}. In the most simple case, there are two possible states: $\S={0,1}.$
	
	The dynamics can be described the following way: At each time step, we choose a vertex uniformly. Said vertex chooses an neighbor also uniformly, and copies its state. Similarly to the Glauber dynamics, we will study the continuous time version instead.
	
	For vertex $i$, the ratio of neighbors sharing belief $s \in \{0,1\}$ is 
	\begin{align*}
	\frac{1}{d(i)}\sum_{j=1}^N a_{ij} \xi_{j,s}(t)=\phi_{i,s}(t)
	\end{align*}	    
	with the choice of Convention 2. Hence, the transition rates take the form
	\begin{align*}
	q_{01}(\phi_i(t))=&\lambda \phi_{i,0}(t), \\
	q_{10}(\phi_{i}(t))=& \lambda \phi_{i,1}(t)=\lambda\left(1-\phi_{i,0}(t)\right).
	\end{align*}
	Using $z_{i,1}(t)=1-z_{i,0}(t)$, NIMFA can be written as
	\begin{align*}
	\frac{\d}{\d t}z_{i,0}(t)=-\lambda (1-\zeta_{i,0}(t))z_{i,0}(t)+\lambda \zeta_{i,0}(t)\left(1-z_{i,0}(t)\right).
	\end{align*}
	
	\subsubsection*{A modified majority rule model}
	
	Another popular model of opinion dynamics is the majority rule \cite{voter_majority}. In this setting a group of $m+1$ individuals are choosen who update their state simultaneously to the majority opinion. Ties are usually broke with either a random choice or setting a preferred opinion to win in this case, say opinion $1$. For the sake of simplicity, we apply the latter approach.
	
	Due to the continuous time setting we use, we modify the majority rule such that only one individual updates its opinion during a transition based on the state of the other vertices (not including its own opinion for the sake of simplicity).
	
	As it is stated in \cite{voter_majority}, the hypergraph setting is more suitable for majority rule. We assume communities have a bounded size $M+1$, while each individual can be a part of many, possibly overlapping communities. 
	
	$a_{i,j_1,\dots,j_m}^{(m)}$ is the indicator of vertices $i, j_1, \dots, j_m \in [N]$ being in a community. We assume symmetry in the indices and set $a_{i,j_1,\dots,j_m}^{(m)}=0$ if there are duplicates. We use a slightly modified version of Convention 1:
	\begin{align*}
	w_{i,\ul{j}}^{(m)}=\frac{\alpha_m a_{i,\ul{j}}}{m! \bar{d}^{(m)}},
	\end{align*}
	where $\alpha_m$ measures how much importance vertices put on communities of size $m+1$. $\remm{\bar{w}_{\textrm{max}}}$ can be small either due to vertices being part of many communities of size $m+1$ on average or because they put less importance on said communicates.
	
	Introduce the notation $|s|=\sum_{l=1}^{m}s_l.$ Vertex $i$ in community $i,j_1,\dots,j_m$ changes its opinion to the majority of $j_1,\dots, j_m$ at rate $w_{i,\ul{j}}^{(m)}$. Therefore,
	\begin{align*}
	q_{01}(\phi_{i}(t))=&\sum_{m=0}^{M} \sum_{\ul{j} \in [N]^m}w_{i,\ul{j}}^{(m)}\1{0 \ \textit{is the majority for $j_1, \dots, j_m$}}\\
	=& \sum_{m=0}^{M} \sum_{\ul{j} \in [N]^m}w_{i,\ul{j}}^{(m)} \sum_{|\ul{s}|<\frac{m}{2}} \prod_{l=1}^{m} \xi_{j_l,s_l}(t)=\sum_{m=0}^{M} \sum_{|\ul{s}|<\frac{m}{2}} \phi_{i,\ul{s}}^{(m)}(t), \\
	q_{10}(\phi_{i}(t))=&\sum_{m=0}^{M} \sum_{|\ul{s}| \geq \frac{m}{2}} \phi_{i,\ul{s}}^{(m)}(t).
	\end{align*}

	The NIMFA ODEs are
	\begin{align*}
		\frac{\d}{\d t}z_{i,0}(t)=&(1-z_{i,0}(t))\sum_{m=0}^{M} \sum_{|\ul{s}|<\frac{m}{2}} \zeta_{i,\ul{s}}^{(m)}(t)-z_{i,0}(t)\sum_{m=0}^{M} \sum_{|\ul{s}| \geq \frac{m}{2}} \zeta_{i,\ul{s}}^{(m)}(t).
	\end{align*}

	\section{Error bounds for NIMFA}
	\label{s:errorbounds}
	
	In this section we are presenting our main results which bound the error arising from neglecting the dynamical correlation between vertices.

	Recall that \eqref{eq:preNIMFA} was closed by assuming $\phi_{i}(t)\approx \E \left( \phi_{i}(t)\right)$. We introduce an auxiliary process where the empirical neighborhood $\phi_{i}(t)$ is replaced by the approximate $\zeta_i(t)$ from \eqref{eq:NIMFA}:
	\begin{align} 
		\label{eq:hat_xi}
		\begin{split}
			\hat{\xi}_{i,s}(t)=&\xi_{i,s}(0)+\sum_{\substack{s' \in \S \\ s' \neq s}} \N_{i,ss'} \left(\K_{i,ss'}(t) \right)-\N_{i,s's} \left(\K_{i,s's}(t) \right),\\
			\K_{i,ss'}(t)=& \left\{(\tau,x) \in \mathbb{R}^2 \left. \right| 0 \leq \tau \leq t, \ 0 \leq x \leq q_{ss'} \left(\zeta_{i}(\tau) \right)  \hat{\xi}_{i,s'}(\tau)\right\}.
		\end{split}
	\end{align}
	The process $\hat{\xi}_{i,s}(t)$ is an indicator process just like $\xi_{i,s}(t)$, so it takes 0 or 1 values, and $\sum_{s\in\S}\hat{\xi}_{i,s}(t)=1$ for any $i\in [N]$ and $t\geq 0$. However, assuming independent initial conditions, $\hat{\xi}_{i}(t)$ remain independent. Applying total expectation to \eqref{eq:hat_xi} shows
	\begin{align*}
		\frac{\d}{\d t}\E \left(\hat{\xi}_{i}(t) \right)= Q \left(\zeta_{i}(t) \right) \E \left(\hat{\xi}_{i}(t) \right),
	\end{align*}
	which, along with \eqref{eq:NIMFA}, implies that if $\E \left(\hat{\xi}_{i}(0) \right)=z_{i}(0)$, then $\hat{\xi}_{i}(t)-z_i(t)$ is a martingale and
	\begin{align}
		\label{NIMFA_initial_condition}
		\E \left(\hat{\xi}_{i}(t) \right)=z_{i}(t) \quad  \forall t \geq 0.
	\end{align}

	Using the same background Poisson processes $ \N_{i,ss'}$ provides a coupling between $\xi$ and $\hat\xi$ that will be useful later on.

	We aim to give an upper bound for $|\hat{\xi}(t)-\xi(t)|$, as well as for $|\hat{\xi}(t)-z(t)|$. We start with $|\hat{\xi}(t)-\xi(t)|$ by introducing the error terms
	\begin{align*}
		D_{i}^{(0)}(t)=& \sup_{0 \leq \tau \leq t} \E \left(\sum_{s \in \S} \left|\xi_{i,s}(\tau)-\hat{\xi}_{i,s}(\tau) \right| \right), \\
		\tilde{D}_{i}^{(0)}(t)=&  \E \left(\sup_{0 \leq \tau \leq t}\sum_{s \in \S} \left|\xi_{i,s}(\tau)-\hat{\xi}_{i,s}(\tau) \right| \right).
	\end{align*}
	Apparently, the only difference between the two is the order in which we take the supremum in time. $\tilde{D}_{i}^{(0)}(t)$ is more strict as
	\begin{align*}
		D_{i}^{(0)}(t) \leq \tilde{D}_{i}^{(0)}(t).
	\end{align*}
	
	Observe that  $\sum_{s \in \S} \left|\xi_{i,s}(\tau)-\hat{\xi}_{i,s}(\tau) \right|$ only has two possible values: $0$ if $\xi_{i}(t)=\hat{\xi}_i(t)$, and $2$ otherwise (as there will be two $s \in \mathcal{S}$ indices where $\xi_{i,s}(t),\hat{\xi}_{i,s}(t)$ differs). This implies
	
	\begin{align*}
		\sup_{0 \leq \tau \leq t}\pr \left( \xi_{i}(\tau) \neq \hat{\xi}_{i}(\tau) \right)=& \frac{1}{2}D_i^{\rem{(0)}}(t), \\
		\pr \left( \exists \ 0 \leq \tau \leq t: \ \xi_{i}(\tau) \neq \hat{\xi}_{i}(\tau) \right)=& \frac{1}{2}\bar{D}_i^{\rem{(0)}}(t) 
	\end{align*}

	We also introduce error terms describing the environments arising from $\xi_{i}(t)$ and $\hat{\xi}_{i}(t)$:
	\begin{align*}
		D_{i}^{(m)}(t)=& \sup_{0 \leq \tau \leq t} \E \left[\sum_{\underline{s} \in \S^m} \left|\phi_{i,\underline{s}}^{(m)}(\tau)-\zeta_{i,\underline{s}}^{(m)}(\tau) \right| \right]\quad (1 \leq m \leq M),\\
		\tilde{D}_{i}^{(m)}(t)=& \, \E \left[\sup_{0 \leq \tau \leq t} \sum_{\underline{s} \in \S^m} \left|\phi_{i,\underline{s}}^{(m)}(\tau)-\zeta_{i,\underline{s}}^{(m)}(\tau) \right| \right]\quad (1 \leq m \leq M).
	\end{align*}
	Since the neighborhoods $\phi_{i}(t)$ and $\zeta_{i}(t)$ are constructed from the indicators $\xi_{i}(t)$ and $\hat{\xi}_{i}(t)$, it is reasonable to expect $\phi_{i}(t)$ and $\zeta_{i}(t)$ to be close to each other -- as long as $\xi_i(t)$ and $\hat{\xi}_{i}(t)$ are also close. To avoid circular reasoning, we carry on handling these two types of errors together at the same time. This motivates the introduction of
	\begin{align*}
		D_{\max}^{(m)}(t)=& \max_{i \in [N]} D_{i}^{(m)}(t), \\
		D_{\max}(t)=& \sum_{m=0}^{M} D_{\max}^{(m)}(t), \\
		\tilde{D}_{i}(t)=& \sum_{m=0}^{M}\tilde{D}_{i}^{(m)}(t).
	\end{align*}
	
	The vector notation $\tilde{D}(t)=\left(\bar{D}_{i}(t) \right)_{i \in [N]}$ will also be utilized.
	
	Now we can go ahead to state the main results of the paper. The idea behind the staments is when the vertex weights are generally small (the network is well-distributed) then vertices has low correlation between each other, hence NIMFA is accurate.
	\begin{theorem} (Main)
		\label{t:Main}
		
		Assume the initial conditions $\xi_{i}(0)$ are independent and \eqref{NIMFA_initial_condition} is satisfied. Then for every $t \geq 0$ there is a constant $C=C\left(t, \delta_{\max},R\right)$ such that
		\begin{align}
			\label{eq:main1}
		\max_{i}	\sup_{0 \leq \tau \leq t}\pr \left( \xi_{i}(\tau) \neq \hat{\xi}_{i}(\tau) \right) \leq 	\frac{1}{2}D_{\max}(t) \leq & C  \sqrt{\remm{w_{\max}^{*}}}.
		\end{align}
		
		Furthermore, if we additionally assume $M=1$ (having $1$-uniform hypergraphs) then there exist constants $C_1=C_1(\delta_{\max}), C_2=C_{2}(\delta_{\max})$ such that for all $t \geq 0$
		\begin{align}
			\label{eq:main2}
			\begin{split}
				\left \|\tilde{D}(t) \right \| \leq& C_1\remm{(1+t)}\exp\left(C_2\left\|W+I\right \|t \right) \|\mu \|,\\
				\mu=&\left( \sqrt{\sum_{j=1}^{N}w_{ij}^2}\right)_{i \in [N]},
			\end{split}
		\end{align}
		where the norm $\| \cdot \|$ is arbitrary, \rem{$W=\left(w_{ij}\right)_{i,j=1}^N$ and $I$ is the identity matrix}.
	\end{theorem}
	
	\begin{remark}The reason why we have different results for $M>1$ and $M=1$ is technical in nature. The main observation is that in the $M=1$ case $\hat{\xi}_{i,s}(t)-z_{i,s}(t)$ \ref{eq:2_norm_bound} \remm{can be easily made to be a martingale by the appropriate compensator} martingale making possible to take $\sup_{0 \leq \tau \leq t}$ inside the expectation via Doob's inequality. It is no longer the case for $M>1$.
	\end{remark}
	
\eqref{eq:main1} is a local result in the sense that it provides a uniform bound, ensuring that $\hat{\xi}_{i,s}(t)$ and $\xi_{i,s}(t)$ are close for all vertices $i$ simultaneously. For example, in the SIS process it allows us to approximate infection probabilities for concrete individuals, not just global or mesoscopic ratios. 
		
	\eqref{eq:main2} will be elaborated on in Theorem \ref{t:extra}. 
	
	In general, we cannot expect a similar local result for $\hat{\xi}_{i,s}(t)$ and $z_{i,s}(t)$ since $\hat{\xi}_{i,s}(t)$ is an indicator while $z_{i,s}(t)$ is a continuous variable. However, if we average out $\hat{\xi}_{i,s}(t)$ over a macroscopic set of vertices, a similar result will hold.
	
	In \eqref{eq:main2} the use of $\ell^2$ or $\ell^{\infty}$ is advised. Observe
		\begin{align}
		\nonumber
		&\|W \|_\infty=\max_{i} \sum_{j}w_{ij} \leq \delta_{\textrm{max}} \\
		\label{eq:2_norm_bound}
		& \|W \|_2 \leq \sqrt{\|W \|_1 \|W \|_{\infty}}=\sqrt{\left(\max_{j} \sum_{i}w_{ij}\right)\left(\max_{j} \sum_{i}w_{ij}\right)} \leq \sqrt{\delta_{\textrm{max}}^{\textrm{out}}\delta_{\textrm{max}}}, 
	\end{align}
	making $\exp\left(C_2\left\|W+I\right \|t \right)$ bounded in \eqref{eq:main2}. Note that \eqref{eq:2_norm_bound} is the only step where Assumption \eqref{eq:out_degree} regarding $\delta_{\textrm{max}}^{\textrm{out}}$ is used.

	As for $\|\mu \|$:
	
	\begin{align*}
		\|\mu \|_{\infty}=&\max_{1 \leq i \leq N} \sqrt{\sum_{j=1}^{n}w_{ij}^2} \leq \max_{1 \leq i \leq N} \sqrt{w_{\textrm{max}}\sum_{j=1}^{n}w_{ij}} \leq \sqrt{w_{\textrm{max}} \delta_{\textrm{max}}},  \\	
		\|\mu \|_{2}=&\sqrt{\sum_{i=1}^{N}\sum_{j=1}^N w_{ij}^2}.
	\end{align*}
	
	Convention 1 works well with the $O \left(\sqrt{w_{\textrm{max}}} \right)$ error bound as $w_{\textrm{max}}=\frac{1}{\bar{d}}$ holds in that case suggesting vertices being close to independent when they have a lot of neighbors on average. Similarly to \eqref{eq:main1}, it also gives a uniform error bound, making it possible to approximate the probabilities at the individual level. For Convention 2 on the other hand, $w_{\textrm{max}}=\frac{1}{d_{\min}}$  is sensitive to even one vertex with a low degree. If we are not attached to uniform bounds in $i,$ we can provide a more robust on for the error of a typical vertex, thus, it is possible to describe global or mesoscopic population statistics.
	
	Let $\iota \sim U \left([N]\right)$ the index of a randomly chosen vertex.
	\begin{align*}
		&\pr \left(  \exists \ \tau \in [0,t]: \ \xi_{\iota}(\tau) \neq \hat{\xi}_{\iota}(\tau) \right) =\\
		&\frac{1}{N}\sum_{i=1}^N\pr \left(  \exists \ \tau \in [0,t]: \ \xi_{i}(\tau) \neq \hat{\xi}_{i}(\tau) \right) \leq\\
		& \frac{1}{2 N}\sum_{i=1}^{N}\tilde{D}_i(t) \leq \sqrt{\frac{1}{4 N}\sum_{i=1}^{N}\tilde{D}_i^2(t)}=O \left( \sqrt{\frac{1}{ N} \|\mu \|_2^2}\right)
	\end{align*}
	Observe
	\begin{align}
		\label{eq:frobenius}
		\frac{1}{N}\| \mu \|_2^2= \frac{1}{N}\sum_{i=1}^N \sum_{j=1}^{N}w_{ij}^2
	\end{align}
	is the squared and normalized Frobenius norm of the matrix $W.$ We mention that such bound were used in \cite{Sridhar_Kar2} under more strict assumptions regarding $W$.

	Note that for Convention 2
		\begin{align}
		\label{eq:main4}
		\frac{1}{N}\sum_{i=1}^N \sum_{j=1}^{N}w_{ij}^2=\frac{1}{N}\sum_{i=1}^N \sum_{j=1}^{N}\frac{a_{ij}}{d^2(i)}=\frac{1}{N}\sum_{i=1}^N \frac{1}{d(i)},
		\end{align}
	meaning the error is small when vertices typically have large degrees.
	
	These observations along with Theorem \ref{t:Main} give the following result:
	\begin{theorem}
		\label{t:extra}
		For $M=1$ (directed, weighted graphs), there exist constants $C_1=C_1(t,\delta_{\textrm{max}})$ and $C_2=C_2(t,\delta_{\textrm{max}},\delta_{\textrm{max}}^{\textrm{out}})$ such that
		\begin{align}
			\label{eq:extra1}
			\max_{i}  \pr \left( \exists \ \tau \in [0,t] : \ \xi_{i}(\tau) \neq \hat{\xi}_{i}(\tau) \right) \leq & C_1 \sqrt{\remm{w_{\max}^*}}, \\
			\label{eq:extra2}
			\frac{1}{N} \sum_{i=1}^N\pr \left( \exists \ \tau \in [0,t] : \ \xi_{i}(\tau) \neq \hat{\xi}_{i}(\tau) \right) \leq & C_2 \sqrt{\frac{1}{N}\sum_{i=1}^N \sum_{j=1}^N w_{ij}^2}.
		\end{align}
	\end{theorem}	
  
	So far, we have only accounted for the error between $\xi_i(t)$ and $\hat{\xi}_i(t)$, however, what we are actually interested in is the expectation $\E\left(\hat{\xi}_i(t)\right)=z_i(t)$, the solution of the ODE system given by NIMFA. Thankfully, $\left(\hat{\xi}_{i}(t) \right)_{i \in [N]}$ are independent, hence, their averages must concentrate around the mean:

	\begin{theorem}
		\label{t:fslln}
		Assume \eqref{NIMFA_initial_condition} holds with independent initial conditions. Then \remm{there is a constant $\bar{C}$ such that}  and any $1\leq K\leq N$, 
		\begin{align}
			\label{eq:fslln}
			\E \left[ \sup_{0 \leq \tau \leq  t} \sum_{s \in \S} \left|\frac1K\sum_{i=1}^K\left(\hat{\xi}_{i,s}(\tau)-z_{i,s}(\tau)\right) \right| \right] \leq  \frac{\remm{\bar{C}(1+t)}}{\sqrt{K}}.
		\end{align}	
	\end{theorem} 
	
	The most natural application of Theorem \ref{t:fslln} is for $K=N$, but it is formulated in a way so that it can be applied to any convenient subset of vertices (the fact that the first $K$ vertices are considered has no significance as the vertices can be reordered arbitrarily).
	
	Together, Theorems \ref{t:Main}, \ref{t:extra} and \ref{t:fslln} give an error bound for the NIMFA approximation.
	\begin{theorem}
		\label{t:xi_z}
		Assume \eqref{NIMFA_initial_condition} holds with independent initial conditions. Then for any $t \geq 0$, there exists a constant $C=C(t,\delta_{\max},R)$ such that
		\begin{align}
			\label{eq:thm4a}
			\sup_{0 \leq \tau \leq t} \E \left(\sum_{s \in \S} \left|\frac1N\sum_{i=1}^N\left(\xi_{i,s}(\tau)-z_{i,s}(\tau)\right) \right| \right)\leq
			C\left(\sqrt{\remm{w_{\max}^{*}}}+\frac{1}{\sqrt{N}}\right).
		\end{align}
		
		Furthermore, if we additionally assume $M=1$, there exist constants $C_1=C_1(t,\delta_{\max}), C_2=C_{2}(t,\delta_{\max},\delta_{\textrm{max}}^{\textrm{out}})$ such that
		\begin{align}
			\label{eq:thm4b}
			\E\left[\sup_{0 \leq \tau \leq t}  \left(\sum_{s \in \S} \left|\frac1N\sum_{i=1}^N\left(\xi_{i,s}(t)-z_{i,s}(t)\right) \right| \right)\right] \leq& C_1\left(\sqrt{\remm{w^{*}_{\max}}}+\frac{1}{\sqrt{N}} \right)
		\end{align}
		and
		\begin{align}
			\label{eq:thm4c}
			\E\left[\sup_{0 \leq \tau \leq t}  \left(\sum_{s \in \S} \left|\frac1N\sum_{i=1}^N\left(\xi_{i,s}(t)-z_{i,s}(t)\right) \right| \right)\right] \leq& C_2\left(\sqrt{\frac{1}{N}\sum_{i=1}^N \sum_{j=1}^N w_{ij}^2}+\frac{1}{\sqrt{N}} \right)
		\end{align}
		where $\mu$ is the same as for Theorem \ref{t:Main}.
	\end{theorem}
	
	\begin{figure}[t]
		\centerline{\includegraphics[scale=0.50]{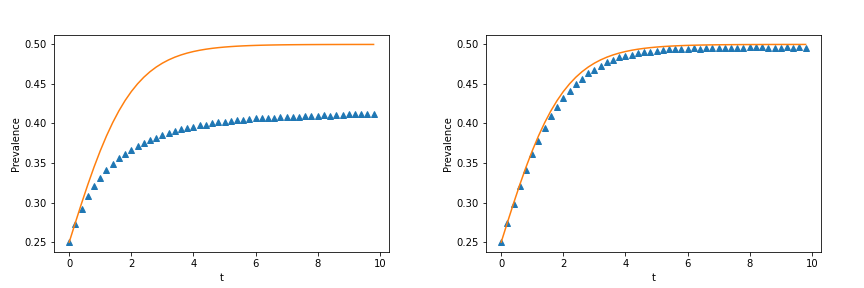}}
		\caption{The ratio of infected based on the average of $1000$ simulations (triangles) compared to the estimate of NIMFA (solid line) on an $N=1000$ vertex modified cycle graphs with the closest $10$ (left) and $100$ (right) neighbors being connected. ($\beta=2, \gamma=1$) As we increase the degrees NIMFA performs better.}
		\label{fig_curves}
	\end{figure}

	\begin{figure}[t]
	\centerline{\includegraphics[scale=0.50]{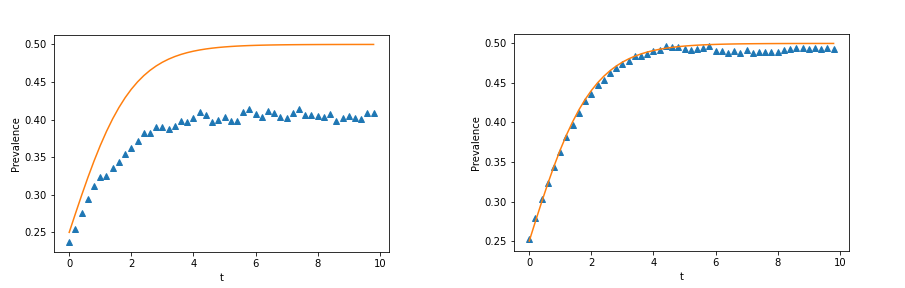}}
	\caption{The ratio of infected based on the average of $10$ simulations (triangles) compared to the estimate of NIMFA (solid line) on an $N=5000$ vertex modified cycle graphs with the closest $10$ (left) and $100$ (right) neighbors being connected. ($\beta=2, \gamma=1$) As we increase the degrees NIMFA performs better.}
	\label{fig_curves2}
\end{figure} 

	\subsection*{Related works}
	
	In this section we compare our results to the recent independent work of Sridhar and Kar \cite{Sridhar_Kar1, Sridhar_Kar2} and Parasnis et al. \cite{age_structured}.
	
	In \cite{Sridhar_Kar1} the authors describe how the state densities of certain related stochastic processes on weighted graphs with doubly symmetric matrix $W$ can be approximated by a set of $O(N)$ ODEs analogous to NIMFA given that the normalized Frobenius norm $\frac{1}{N}\sum_{i=1}\sum_{j=1}^N w_{ij}^2$ is small and $N$ is large.
	
	Given the conclusions of Theorem 4.2 in \cite{Sridhar_Kar1} and Theorem \ref{t:xi_z} in the present paper are very similar in nature, it makes sense to compare the general setup, the conditions, the conclusions and the technique directly to those in the present paper.
	
	
	
	Setup. Strictly speaking, the stochastic processes discussed in the present paper and in \cite{Sridhar_Kar1, Sridhar_Kar2} are different. In our work, time is continuous while \cite{Sridhar_Kar1} and \cite{Sridhar_Kar2} start from discrete time steps then speed up time. This is a minor difference though, and with appropriate time scaling, the models in \cite{Sridhar_Kar1, Sridhar_Kar2} and the present paper define essentially the same object.
	
	Conditions. In the present paper, we require only that the normalized degrees are bounded. This is more general than the doubly stochastic $W$ assumption of \cite{Sridhar_Kar1, Sridhar_Kar2}. Specifically, our result also justifies Example 4.2 in \cite{ Sridhar_Kar2}.

	Via \eqref{eq:thm4c}, qualitatively the same type of error terms were retained in terms of the normalized Frobenius norm, but \cite{Sridhar_Kar1, Sridhar_Kar2} provides an error probability bound that is exponential in $N$. In the present paper, we do not focus on this kind of large deviation bound in $N$.
	
	
	\cite{Sridhar_Kar1, Sridhar_Kar2} derive bounds for the global average. On the other hand, our results show more localized, uniform bounds in terms of vertices. This is made possible by the use of the auxiliary Markov processes $\hat{\xi}_{i}(t)$, allowing accurate predictions about individual vertices too, not just global averages. 
	
	Our framework also allows higher order interactions, while \cite{Sridhar_Kar1, Sridhar_Kar2} is restricted to order 2 interactions (graphs).
	
	In \cite{age_structured} the authors study the SIR process in age-structured populations on time-varying networks. They show that when $N$ and the rewiring rate is high the prevalence of the age groups can be described via an ODE system analogous to the metapopulation NIMFA model \eqref{eq:meta_NIFMA} in Section \ref{s:meta}. Note that \cite{age_structured} applies to cases with fast, but finite rewiring rates as well, while our result only considers the idealized case of infinite rewiring rates.

	\section{Further reductions to NIMFA}
	\label{s:red}
	
	This section relates NIMFA to other approaches from the literature. Although NIMFA is a major reduction of the exact Kolmogorov-equations, requiring only $O(N)$ ODEs to be solved, it can be still computationally prohibitive when the number of vertices is too large. Furthermore, NIMFA requires knowing both the full network structure and precise initial conditions for all vertices. We look at further reductions to \eqref{eq:NIMFA} when additional structure is known for the network or initial conditions; several of these actually lead to other well-known models from the literature.
	
	\subsection{Homogeneous mean field approximation}
	\label{s:HMFA}
	
	The homogeneous mean field approximation (HMFA) assumes that the vertices are \emph{well mixed}, meaning, every vertex interacts with every other with equal weights. Formally, this can be this can be described by a complete hypergraph (with all loops and secondary loops):
	\begin{align*}
		w_{i,\underline{j}}^{(m)}=\frac{1}{N^m}.
	\end{align*}
	
	This definition may be generalized to include cases when $w_{i, \underline{j}}^{(m)}=0$ for certain $m$ indices, e.g. $(M+1)$-uniform hypergraphs. For ease of notation, instead of modifying the definition of $w_{i, \underline{j}}^{(m)}$, it is also possible to choose the rate functions $q_{ss'}(\phi)$ so that they do not depend on the appropriate $\phi^{(m)}$ coordinates, making the choice of $w_{i,\underline{j}}^{(m)}$ irrelevant.
	
	\remm{
	
	\begin{remark}
	\label{r:w_bar}
	Let $\xi_1, \dots, \xi_m, \eta_1, \dots, \eta_m$ be i.i.d. uniform variables from the index set $[N]$. Then 
	\begin{align*}
	\frac{1}{\left(N^m \right)^2} \sum_{\substack{\ul{i}, \ul{j} \in [N]^m \\ \ul{j} \cap \ul{k} \neq \emptyset}} 1=\pr \left( \ul{\xi} \cap \ul{\eta} \neq \emptyset \right) \leq & \sum_{j,k=1}^{m} \pr \left(\xi_{j} =\eta_{k} \right)=\frac{m^2}{N}, \\
	\geq & \pr \left(\xi_{1}=\eta_{1} \right)=\frac{1}{N},
	\end{align*}
	
implying
\begin{align*}
	\frac{1}{\left(N^m \right)^2} \sum_{\substack{\ul{i}, \ul{j} \in [N]^m \\ \ul{j} \cap \ul{k} \neq \emptyset}} 1 \asymp \frac{1}{N} .
\end{align*}
	\end{remark}
	}

	For such networks, $\remm{w_{\max}^{*} \asymp}\frac{1}{N}$ and $\delta_{\max}=1$. What remains to show is that \eqref{A:2} holds with some bounded $R$.
	\begin{align}
		\label{eq:diag_bound}
		\begin{split}
			&\sum_{\substack{\underline{j} \in [N]^m \\ \underline{j} \textit{ is s.\ loop}}}w_{i, \underline{j}}^{(m)}= \frac{1}{N^m} \left|\left \{\left. \underline{j} \in [N]^m \right| \underline{j} \textrm{ s.\ loop} \right\} \right|=\\
			&1-\frac{1}{N^m} \left|\left \{\left. \underline{j} \in [N]^m \right| \underline{j} \textrm{ not s. loop} \right\} \right|=1-\prod_{l=0}^{m-1}\left(1-\frac{l}{N} \right)=\\
			&O \left(\frac{1}{N}\right) \ll \sqrt{\remm{w_{\max}^{*}}},
		\end{split}
	\end{align}
	hence, $R$ can be chosen arbitrarily small for large enough $N$.
	
	Our goal now is to derive a small system of equations for
	\begin{align*}
		u(t):=\frac{1}{N}\sum_{i=1}^N z_{i}(t).
	\end{align*}
	Our strategy is based on the observation that the neighbourhood vectors $\zeta_{i}(t)$ are the same for all vertices.
	\begin{align*}
		\zeta_{i,\underline{s}}^{(m)}(t)=&\frac{1}{N^m}\sum_{\underline{j} \in [N]^m} \prod_{l=1}^{m}z_{j_l,s_l}(t)=\prod_{l=1}^m \left(\frac{1}{N}\sum_{j_l=1}^N z_{j_l,s_l}(t) \right)=\\
		&\prod_{l=1}^{m}u_{s_l}(t)=:u_{\underline{s}}^{(m)}(t)
	\end{align*}
	
	This results in the ODE system: 
	\begin{align}
		\label{eq:HMFA}
		\begin{split}
			\frac{\d}{\d t}u(t)=& Q \left(U(t) \right)u(t),\\
			U(t)=& \left(u^{(m)}(t) \right)_{m=1}^M, \\
			u^{(m)}(t)=& \left(u_{\underline{s}}^{(m)}(t) \right)_{\underline{s}\in \S^m}=\left(\prod_{l=1}^{m}u_{s_l}(t) \right)_{\underline{s}\in \S^m}.
		\end{split}
	\end{align}
	
	For example, the simplicial SIS model \eqref{eq:HMFA} takes the form
	\begin{align*}
		\frac{\d}{\d t} u_{I}(t)=-\gamma u_{I}(t)+\left(1-u_{I}(t) \right)\sum_{m=1}^{M} \beta_{m}u_{I}^{m}(t). 
	\end{align*}
	which was used in \cite{simplicial2}.

	In this setting, Theorem \ref{t:xi_z} shows the ratio of vertices in state $s \in \mathcal{S}$ can be approximated by $u_s(t)$ with $O \left( \frac{1}{\sqrt{N}} \right)$ error. The well known results of Kurtz \cite{kurtz70, kurtz78} correspond to the $M=1$ case.
	
	\subsubsection*{Regular hypergraphs}
	
	Although \eqref{eq:HMFA} is both feasible for analytical and numerical investigations (due to its finite size) the assumption that the network structure is well-mixed is quite restrictive. However, as we will see, the well-mixed condition can be relaxed given uniform initial conditions.
	
	We call a weighted hypergraph \emph{regular} if
	\begin{align}
		\label{def:regular}
		\forall \ 1 \leq i \leq N, \ 1 \leq m \leq M  \ \ \delta^{(m)}(i)=1.
	\end{align}
	Note that the value $1$ is arbitrary and any other constant value would work with minor modifications to the rate functions $q_{ss'}$.
	
	We note that \eqref{def:regular} always holds for Convention 2 hypergraphs. For Convention 1, it holds when $d^{(m)}(i)=\bar{d}^{(m)}\,\, \forall 1 \leq i \leq N, \ 1 \leq m \leq M$  (that is, the hypergraph is regular in the usual sense).
	
	\begin{prop}
		\label{p:uniform}
		
		Assume \eqref{def:regular} and
		$$ z_{i}(0)=u(0)\quad \forall \ 1 \leq i \leq N$$ for some $u(0) \in \Delta^{\S}.$ Then the solution of \eqref{eq:NIMFA} takes the form
		$$\ z_{i}(t)=u(t)\quad \forall \ 1 \leq i \leq N $$
		where $u(t)$ satisfies \eqref{eq:HMFA}. 
	\end{prop}
	We mention that statements similar to Proposition \ref{p:uniform} have appeared in the literature before in certain special cases \cite[Proposition 3.18 ]{Simon_book}. Combining Proposition \ref{p:uniform} with Theorem \ref{t:Main} ensures the accuracy of the homogeneous mean field approximation on regular graphs with large degrees and homogeneous initial conditions disregarding any further network structure.
	
	\begin{proof}(Proposition \ref{p:uniform})
		
		Let $u(t)$ be the solution of \eqref{eq:HMFA}. Set $z_{i}(t)=u(t).$ We have to show that $z_i(t)$ satisfies \eqref{eq:NIMFA}. The initial conditions are satisfied according to the assumption, and for the derivatives,
		\begin{align*}
			u_{\underline{s}}^{(m)}(t)=&u_{\underline{s}}^{(m)}(t)  \delta^{(m)}(i)=u_{\underline{s}}^{(m)}(t) \sum_{\underline{j} \in [N]^m} w_{i,\underline{j}}^{(m)}=\sum_{\underline{j} \in [N]^m} w_{i,\underline{j}}^{(m)}z_{\underline{j},\underline{s}}^{(m)}(t)=\zeta_{i,\underline{s}}^{(m)}(t),\\
			\frac{\d}{\d t}z_{i}(t)=&\frac{\d }{\d t} u(t)=Q \left(U(t) \right)u(t)=Q \left(\zeta_{i}(t)\right)z_{i}(t).
		\end{align*}
	\end{proof}
	
	\subsection{Metapopulation models}
	\label{s:meta}
	
	As we saw in Section \ref{s:HMFA} \rem{,} a way to reduce the number of equations is by grouping vertices together and representing them by a single averaged-out term. In practice, this approach will only work if the vertices grouped together are sufficiently homogeneous, which is typically not the case for the entire population. To mitigate this issue, we may introduce \emph{communities}, inside which we assume homogeneity, then derive the dynamics between communities. This "higher resolution" may increase accuracy, at the cost of a larger ODE system. 
	
	In practice, the communities can be chosen by demographic and geographic criterion such as age and location. Alternatively, it is also possible to group vertices according to degree, or a third option is the use of community detection algorithms \cite{SzemerediAlgorithm}.
	
	We present the general setup for metapopulation models first for graphs in Section \ref{sss: metagraph}, then for hypergraphs in Section \ref{sss: metahypergraph}\rem{.}
	
	For the SIS process on graphs similar results had been derived in \cite{SISpartition}.
	\subsubsection{Metapopulation models on graphs}
	\label{sss: metagraph}
	
	First, assume $M=1$. Divide the vertices into a partition $V_1,\dots,V_{K}$  with size $\left|V_k\right|=N_k$ such that vertices inside a group are similar in some sense. The average weight between group $V_k$ and $V_l$ is
	\begin{align}
		\label{def:tilde_w1}
		\tilde{w}_{kl}=\frac{\sum_{i \in V_k}\sum_{j \in V_l} w_{ij}}{N_k N_l}.
	\end{align}
	(In the idealized case of metapopulations, $w_{ij}$ would have the same value $\tilde{w}_{kl}$ for each $i \in V_k, j \in V_l$ pair.)
	
	Next we derive the dynamics for the averages
	\begin{align}
		\label{def:bar_z1}
		\bar{z}_{k}(t):=\frac{1}{N_k}\sum_{i \in V_k}  z_{j}(t).
	\end{align}
	
	$\zeta_i(t)$ has the same value $\bar{\zeta}_k(t)$ for all $i \in V_k$:
	\begin{align}
		\label{def:pre_theta1}
		\bar{\zeta}_{k}(t)=\zeta_{i}(t)=\sum_{j=1}^N w_{ij}z_{j}(t)= \sum_{l=1}^K \underbrace{N_l\tilde{w}_{kl}}_{\bar{w}_{kl}} \frac{1}{N_l}\sum_{j \in V_l}z_j(t)=\sum_{l=1}^K \bar{w}_{kl}\bar{z}_l(t).
	\end{align}
	Therefore, we can derive an ODE system for \eqref{def:bar_z1} 
	\begin{align}
		\label{eq:meta_NIFMA}
		\frac{\d}{\d t}\bar{z}_k(t)=Q\left(\bar{\zeta}_k(t) \right)\bar{z}_{k}(t)
	\end{align}
	which is equivalent to \eqref{eq:NIMFA} on the graph $\overline{\mathcal{G}}$ with vertex set $\{1, \dots, K\}$ and weights $\left(\bar{w}_{kl}\right)_{k,l=1}^K.$

	\subsubsection{Metapopulation models on hypergraphs}
	\label{sss: metahypergraph}
	
	For the general metapopulation setting, we assume that for each $m=1,\dots,M$, the population is partitioned into \emph{local groups} $V_1^{(m)},\dots,V_{K^{(m)}}^{(m)}$. The \emph{type} of a vertex will be denoted by $k=\left(k^{(1)}, \dots, k^{(M)} \right)$, which means that for each $m=1,\dots,M$, the given vertex is in the local group $V_{k^{(m)}}^{(m)}$. Vertices can be partitioned according to their type into $\prod_{m=1}^M K^{(m)}$ \emph{global groups}.
	
	We aim to define a hypergraph on the types, with weights consistent with the average of weights within each group. That said, with the above setup, this is easier to do using local groups for each $m=1,\dots,M$.
	
	For a given $m$, $k^{(m)}$ and $\underline{l}^{(m)}=\left(l_1^{(m)}, \dots, l_m^{(m)} \right)$, the \emph{total local $m$-weight between $k^{(m)}$ and $\underline{l}^{(m)}$} is defined as
	\begin{align}
		W_{k^{(m)},\underline{l}^{(m)}}^{(m)}:=\sum_{i \in V_{k^{(m)}}^{(m)}}\sum_{j_1 \in V_{l_1^{(m)}}^{(m)}} \dots \sum_{j_m \in V_{l_m^{(m)}}^{(m)}}w_{i,\underline{j}}^{(m)}.
	\end{align}
	
	Then, using the notation
	$$N_{\underline{l}^{(m)}}:=\prod_{r=1}^{M}N_{l_r^{(m)}}^{(m)},$$
	we define the weight of the edge containing the local groups $k^{(m)}, \underline{l}^{(m)}$ as
	\begin{align}
		\tilde{w}_{k^{(m)},\underline{l}^{(m)}}^{(m)}:=\frac{W_{k^{(m)},\underline{l}^{(m)}}^{(m)}}{N_{k^{(m)}}N_{\underline{l}^{(m)}}}.
	\end{align}
	
	Let $k(i)=\left(k^{(1)}(i), \dots, k^{(M)}(i) \right)$ denote the type of $i$. For easier notation, we will often use $\iota \sim U \left([N] \right)$, which is a random vertex independent from everything else. Then we define the average of $z_i(t)$ over type $k$ as
	\begin{align}
		\label{def:bar_z2}
		\bar{z}_{k}(t):= \E \left(\left.z_{\iota}(t) \right| k(\iota)=k \right)=\frac{1}{N_k}\sum_{i \in V_k}z_{i}(t).	
	\end{align}
	
	In this case as well, $\zeta_i(t)$ has the same value for all $i \in V_k$; this common value will be denoted by $\bar{\zeta}_k(t).$ Let $\iota_1, \dots, \iota_m$ denote i.i.d. copies of $\iota.$ Then
	\begin{align}
		\label{eq:zeta_bar}
		\begin{split}
			\bar{\zeta}_k^{(m)}(t)=& \zeta_i^{(m)}(t)=\sum_{j \in [N]^m}w_{i, \underline{j}}^{(m)}z_{\underline{j}}^{(m)}(t)=\sum_{\underline{l}^{(m)}}\tilde{w}_{k,\underline{l}^{(m)}}^{(m)}\sum_{j_1 \in V_{l_1^{(m)}}^{(m)}} \dots \sum_{j_m \in V_{l_m^{(m)}}^{(m)}}z_{\underline{j}}^{(m)}(t) \\
			=& \sum_{\underline{l}^{(m)}} \underbrace{N_{\underline{l}^{(m)}}\tilde{w}_{k,\underline{l}^{(m)}}^{(m)}}_{:=\bar{w}_{k^{(m)},\underline{j}^{(m)}}^{(m)}} \E \left(\left. \prod_{r=1}^{m}z_{\iota_r}(t) \right| k^{(m)}(\iota_1)=l_1^{(m)}, \dots, k^{(m)}(\iota_{m})=l_m^{(m)} \right)=\\
			=& \sum_{\underline{l}^{(m)}} \bar{w}_{k^{(m)},\underline{j}^{(m)}}^{(m)}\prod_{r=1}^{m}\E \left(\left. z_{\iota}(t) \right| k^{(m)}(\iota)=l_r^{(m)} \right)\rem{.}
		\end{split}
	\end{align}
	
	This means that the ODE system for \eqref{def:bar_z2} is formally the same as \eqref{eq:meta_NIFMA} (with the appropriate definition of $\bar{z}_k(t)$ and $\bar{\zeta}_k(t)$)\rem{.}
	
	Note that $\bar{\zeta}_k(t)$ can also be expressed via $\bar{z}_{k}(t)$ as 
	\begin{align*}
		\E \left(\left. z_{\iota}(t) \right| k^{(m)}(\iota)=l_r^{(m)} \right)=&\E \left(\left. \E \left(\left.z_{\iota}(t) \right| k(i)=k \right) \right| k^{(m)}(\iota)=l_r^{(m)} \right)=\\
		&\E \left(\left. \bar{z}_{k(\iota)}(t) \right| k^{(m)}(\iota)=l_r^{(m)} \right),
	\end{align*}
	making \eqref{eq:meta_NIFMA} a closed system.
	
	In the special case when the hypergraph is $(M+1)$-uniform, we can set $K^{(m)}=1$ for all $m<M$ virtually making the local group $k^{(M)}$ and the global group $k$ the same (apart from some $1$'s in the first $M-1$ components). In this case, $Q$ only depends on $\bar{\zeta}^{(M)}(t)$ which can be expressed as
	\begin{align*}
		\bar{\zeta}_{k^{(M)}}^{(M)}=\sum_{\underline{l}^{(m)}}\bar{w}_{k^{(m)},\underline{l}^{(m)}}^{(m)}\prod_{r=1}^{m}\bar{z}_{k^{(M)}(l_r)}(t).
	\end{align*}

	\subsection{Annealed networks}
	
	So far, we only focused on the dynamics of the Markov process neglecting the dynamics of the network itself. When there is a separation of scale between the speed of the Markov process and the changes to the network itself, two kinds of idealizations are typically used: 
	\begin{itemize}
		\item \emph{quenched networks}: the speed at which the network changes is much slower than the Markov process. In this case, the network is assumed constant in time.
		\item \emph{annealed networks}: the speed at which the network changes is much faster than the Markov process. In this case, we consider the network changes averaged out for the interactions.
	\end{itemize}
	
	Annealed networks can be modeled by replacing connections $a_{i, \underline{j}}^{(m)}$ in \eqref{eq:conv1} and \eqref{eq:conv2} with the average $\langle a \rangle_{i, \underline{j}}^{(m)}$.
	
	In this section, we present a setup for annealed networks generated via the configuration model \cite{configuration}. Similar calculations can be made for other models that include e.g.\ degree correlation  such as equation (93) in \cite{annealed_formula}.
	
	Once again, we start with the graph case.
	
	In the configuration model the degrees $d(1), \dots, d(N)$ are given beforehand, and vertex $i$ receives $d(i)$ half-edges (\emph{stubs}) initially. Then in each round, we choose two stubs at random to connect and form an edge, repeating this procedure until all stubs are paired.
	
	Loops and multiple edges are possible, but their effect will be neglected. The expected connection between vertices $i$ and $j$ is
	\begin{align*}
		\langle a \rangle_{ij}=\frac{d(i)d(j)}{\bar{d}N}.
	\end{align*}
	
	The degree of each vertex $i$ indeed matches the prescribed $d(i)$ as
	\begin{align*}
		\sum_{j=1}^N \langle a \rangle_{ij}=\frac{d(i)}{\bar{d}} \frac{1}{N}\sum_{j=1}^N d(j)=d(i).
	\end{align*}
	
	$\langle a \rangle_{ij}$ depends only on the degrees of $i$ and $j$, so it can be interpreted as a metapopulation model where vertices are grouped according to their degree. (Note that here we also use the index $k=0$ for isolated vertices if any.) The corresponding weights are
	\begin{align*}
		\tilde{w}_{kl}=\frac{kl}{\bar{d}^2 N},
	\end{align*}
	for Convention 1, and 
	\begin{align*}
		\tilde{w}_{kl}=\frac{l}{\bar{d}N}.
	\end{align*}
	for Convention 2.
	
	Let $q_k:= \frac{k N_k}{\bar{d} N}$ denote the size biased degree distribution and introduce
	\begin{align}
		\label{def:theta1}
		\Theta(t):=\sum_{l=0}^{d_{\max}} q_{l}\bar{z}_{l}(t).
	\end{align}
	
	Using \eqref{def:pre_theta1}, $\bar{\zeta}_k(t)$ can be written as
	
	\begin{align*}
		\bar{\zeta}_k(t)= \frac{k}{\bar{d}} \Theta(t),
	\end{align*}
	for Convention 1, and 
	\begin{align*}
		\bar{\zeta}_k(t)=  \Theta(t).
	\end{align*}
	for Convention 2.
	
	For example, the I component of the SIS process assuming Convention 1 is
	\begin{align}
		\label{eq:Vesp}
		\begin{split}
			\frac{\d}{\d t}\bar{z}_{k,I}(t)&= -\gamma \bar{z}_{k,I}(t)+\frac{\beta}{\bar{\bar{d}}} k \left(1-\bar{z}_{k,I}(t) \right)\Theta_I(t) \rem{,} \\ 
			\Theta_I(t)&=\sum_{l=0}^{d_{\textrm{max}}} q_{l}\bar{z}_{l,I}(t) \rem{.}
		\end{split}
	\end{align}
	which is the Inhomogeneous Mean Field Approximation (IMFA) studied by Pastor-Satorras and Vespignani \cite{vesp}.
	
	For Convention 1, to apply the results of the present paper, we need to assume upper regularity, i.e. $\delta_{\max}=\frac{d_{\max}}{\bar{d}}$ to be bounded. In many applications, the degree distribution converges to a fixed distribution, making $\bar{d}$ bounded; in such a setting, we accordingly require $d_{\max}$ to be bounded as well.
	
	Assuming upper regularity,
	\begin{align*}
		w_{\max}= \frac{d_{\max}^2}{\bar{d}^2 N}=\frac{1}{N}\delta_{\max}^2
	\end{align*}
	thus Theorem \ref{t:xi_z} actually provides an $O\left( \frac{1}{\sqrt{N}}\right)$ error bound.
	
	As for Convention 2, $\delta_{\max}=1$ holds as usual, and
	\begin{align*}
		w_{\max}=\frac{1}{N}\frac{d_{\max}}{\bar{d}}.
	\end{align*}
	
	Unfortunately, one can not relax the bound on $d_{\textrm{max}}$ by using \eqref{eq:extra2} instead of \eqref{eq:extra1} as it requires bounds for the out-degrees:
		\begin{align*}
		 \delta^{\textrm{out}}(j)=\sum_{i=1}^N w_{ij}= \sum_{k=0}^{d_{\textrm{max}}} N_k \frac{d(j)}{\bar{d}N}=\frac{d(j)}{\bar{d}} \leq \frac{d_{\textrm{max}}}{\bar{d}} \leq \delta_{\textrm{max}}^{\textrm{out}}.
		\end{align*}

	
	Now we turn to the hypergraph case $M>1$. We generalize the notion of the configuration model in the following manner: For a fixed $m$, the $m$-degrees are given as $d^{(m)}(1), \dots, d^{(m)}(N)$ and each vertex receives $m$-stubs based on their degree. In each round, we choose $m+1$ $m$-stubs at random to form an $m$-edge, then repeat this procedure until all of the stubs have been paired. This procedure is performed for each $1 \leq m \leq M$ independently.
	
	For distinct $i, j_1, \dots j_m$, the probability of connecting them in a given round is
	\begin{align*}
		\frac{d^{(m)}(i) \prod_{r=1}^m d^{(m)}(j_r)}{{\bar{d}^{(m)}N \choose m+1}} \approx \frac{(m+1)! d^{(m)}(i) \prod_{r=1}^m d^{(m)}(j_r)}{\left(\bar{d}^{(m)}N\right)^{m+1}}.
	\end{align*} 
	Since there are $\frac{\bar{d}^{(m)}N}{m+1}$ rounds in total, we set
	\begin{align*}
		\langle a \rangle_{i,\underline{j}}^{(m)}:=\frac{m! d^{(m)}(i) \prod_{r=1}^m d^{(m)}(j_r)}{\left(\bar{d}^{(m)}N\right)^{m}}.
	\end{align*}
	
	For the hypergraph case, we only examine Convention 1, for which
	\begin{align*}
		\tilde{w}_{k^{(m)}, \underline{l}^{(m)} }^{(m)}=\frac{k^{(m)}}{\bar{d}^{(m)}}\frac{  \prod_{r=1}^m l_r^{(m)} }{\left(\bar{d}^{(m)}N\right)^{m}}.
	\end{align*}
	
	Once again, the resulting hypergraph can be interpreted as a metapopulation model, where the local groups are given according to the $m$-degrees of the vertices.
	
	Clearly $\delta^{(m)}(i)= \frac{d^{(m)}(i)}{\bar{d}^{(m)}},$ so we make an upper regularity assumption in this case as well, from which $w_{\max}=O \left( \frac{1}{N} \right)$ follows.
	
	For hypergraphs,
	
	$$\remm{ \frac{1}{\left(\bar{d}^{(m)} \right)^{m+1}N^m}} \leq \tilde{w}_{k^{(m)}, \underline{l}^{(m)} }^{(m)} \leq \frac{\delta_{\max}^{m+1}}{N^m},$$ 
	so \eqref{eq:diag_bound}
	\remm{making $w_{\max}^{*} \asymp \frac{1}{N}$ according to Remark \ref{r:w_bar}. As for condition \eqref{A:2}}
	\begin{align}
		\label{eq:diag_bound2}
		\sum_{\substack{\underline{j} \in [N]^m \\ \underline{j} \textit{ is s.\ loop}}}w_{i, \underline{j}}^{(m)} \leq C \sum_{\substack{\underline{j} \in [N]^m \\ \underline{j} \textit{ is s.\ loop}}} \frac{1}{N^m} =O \left( \frac{1}{N}\right) \ll \sqrt{ \remm{ w_{\max}^{*}}},
	\end{align}
	hence arbitrarily small $R$ can be used for large enough $N$.
	
	The next step is to calculate $\bar{\zeta}_k(t)$ based on \eqref{eq:meta_NIFMA}. Define
	$$q_{k^{(m)}}^{(m)}:= \frac{k^{(m)}N_{k^{(m)}}}{ \bar{d}^{(m)} N},$$ 
	the size-biased degree distribution of the $m$-vertices. Also define
	\begin{align}
		\label{def:Theta2}
		\Theta^{(m)}(t):= \sum_{l=1}^{d_{\max}^{(m)}}q_{l}^{(m)}\E \left(\left.z_{\iota}(t) \right| d^{(m)}(\iota)=l \right),
	\end{align}
	once again using the notation $\iota \sim U \left([N] \right)$.
	
	Using \eqref{eq:zeta_bar},
	\begin{align*}
		\bar{\zeta}_{k}^{(m)}(t)&=\sum_{\underline{l}^{(m)}} \bar{w}_{k^{(m)},\underline{j}^{(m)}}^{(m)}\prod_{r=1}^{m}\E \left(\left. z_{\iota}(t) \right| d^{(m)}(\iota)=l_r^{(m)} \right)=\\
		&= \frac{k^{(m)}}{\bar{d}^{(m)}}\sum_{\underline{l}^{(m)}} \prod_{r=1}^{m}q_{l_r}^{(m)}\E \left(\left.z_{\iota}(t) \right| d^{(m)}(\iota)=l_r \right)\\
		&= \frac{k^{(m)}}{\bar{d}^{(m)}} \prod_{r=1}^{m}\sum_{l_r=1}^{d_{\textrm{max}}^{(m)}}q_{l_r}^{(m)}\E \left(\left.z_{\iota}(t) \right| d^{(m)}(\iota)=l_r \right)=\\
		&= \frac{k^{(m)}}{\bar{d}^{(m)}} \left(\Theta^{(m)}(t) \right)^m.
	\end{align*}
	
	Accordingly, e.g. the dynamics for the simplicial SIS model can be written as
	\begin{align}
		\label{eq:IHMFA_simplicial}
		\frac{\d}{\d t}\bar{z}_{k,I}(t)=-\gamma \bar{z}_{k,I}(t)+(1-\bar{z}_{k,I}(t))\sum_{m=1}^{M} \frac{\beta^{(m)}}{\bar{d}^{(m)}}\left(\Theta_{I}^{(m)}(t) \right)^{m} \rem{.}
	\end{align}
	
	\eqref{eq:IHMFA_simplicial} was studied in \cite{simplicial3} for the $(M+1)$-uniform case, where $\E \left(\left.z_{\iota}(t) \right| d^{(M)}(\iota)=l \right)$ simplifies to $\bar{z}_{k}(t)$ as the global class $k$ and the local class $k^{(M)}$ coincide. 
	
	\subsection{Activity-driven networks}
	
	Activity-driven networks  were introduced in  \cite{activity1}.
	
	Let $a_1, \dots, a_K$ be positive numbers called \emph{activities} and let $a(i)$ denote the activity of vertex $i$. Instead of a graphs structure, each vertex chooses a random vertex uniformly with rate $\beta a(i)$ and if they are an SI pair, the susceptible node becomes infected. Recoveries happen independently with rate $\gamma.$
	
	The above model corresponds to an SIS process on the weighted graph
	\begin{align*}
		w_{ij}=\frac{a(i)+a(j)}{N}
	\end{align*} 
	since  to form the $(i,j)$ pair, either $i$ or $j$ needs to activate, and each vertex is chosen with probability $\frac{1}{N}.$ The graph is a metapopulation model, with groups corresponding to the activity values.
	
	We generalize this concept to allow higher order interactions. $a_{1}^{(m)}, \dots, a_{K^{(m)}}^{(m)}$ are the possible $m$-activities and we assume that vertex $i$ chooses $m$ other vertices at random with rate $a^{(m)}(i).$ This results in a hypergraph with weights
	\begin{align*}
		w_{i,\underline{j}}^{(m)}= \frac{1}{N^m}\left(a_{i}^{(m)}+\sum_{r=1}^{m}a_{j_r}^{(m)} \right).
	\end{align*}
	
	Assume the activity rates are bounded from above by some $a_{\max}<\infty.$ Also, introduce
	$$\bar{a}^{(m)}:=\frac{1}{N}\sum_{i=1}^{N}a^{(m)}(i). $$
	Then
	\begin{align*}
		\delta^{(m)}(i)= a_{i}^{(m)} +\frac{1}{N^m}\sum_{\underline{j} \in [N]^m} \sum_{r=1}^{m}a_{j_r}^{(m)} =a_{i}^{(m)}+\bar{a}^{(m)} \leq 2a_{\max}
	\end{align*}
	so \eqref{A:1} is satisfied.
	
	\remm{Clearly,
	\begin{align*}
	\frac{a_{i}^{(m)}}{N^{m}} \leq w_{i, \ul{j}}^{(m)} \leq \frac{(m+1)a_{max}}{N^{m}}
	\end{align*}
	}
	\remm{ making $w_{\max}^{*} \asymp \frac{1}{N}$.}  \eqref{eq:diag_bound2} is applicable here as well satisfying \eqref{A:2}, hence Theorem \ref{t:Main} applies.
	
	$\bar{\zeta}_k(t)$ can also be expressed with the help of \eqref{eq:zeta_bar}.
	\begin{prop}
		\label{p:activity}
		Let $\iota \sim U([N])$ a random index and $p^{(m)}_{k^{(m)}}$ be the ratio of vertices in the local group $k^{(m)}.$ Also, define
		\begin{align*}
			\psi^{(m)}(t):= \sum_{l=1}^{K^{(m)}}a_{l}^{(m)}p_{l}^{(m)}\E \left(\left. z_{\iota}(t) \right| a^{(m)}(\iota)=l^{(m)} \right).
		\end{align*}
		
		Then the neighborhood vectors have the form
		\begin{align*}	
			\bar{\zeta}^{(m)}_k(t)=\left(a_{k^m}^{(m)}\E \left(z_{\iota}(t)\right)+\psi^{(m)}(t) \right)\E^{m-1} \left(z_{\iota}(t)\right).
		\end{align*}	
	\end{prop}
	
	The proof of Proposition \ref{p:activity} is given in Section \ref{s:proofs}.
	
	For activity-driven networks, the simplicial SIS model takes the form
	\begin{align}
		\begin{split}
			\label{eq:activity_SIS}
			\frac{\d}{\d t} \bar{z}_{k,I}(t)=&-\gamma\bar{z}_{k,I}(t)+
			\left(1-\bar{z}_{k,I}(t) \right) \cdot \\
			&\sum_{m=1}^{M} \beta_{m}\E^{m-1} \left(z_{\iota,I}(t)\right)\left(a_{k^m}^{(m)}\E \left(z_{\iota,I}(t)\right)+\psi_I^{(m)}(t) \right).
		\end{split}
	\end{align}
	\cite{activity2} proves that \eqref{eq:activity_SIS} describes the large graph limit correctly when $M=1$.

	\subsection{Dense graphs and Szemer\'edi's regularity lemma}
	
	We call a hypegraph dense if there is some $0<p_0 \leq 1$ such that
	\begin{align}
		\label{def:dense}
		\bar{d}^{(m)} \geq p_0N^m\quad \forall \ 1 \leq m \leq M.
	\end{align}
	
	For Convention 1 graphs,
	\begin{align*}
		\frac{1}{M! \ N}\leq w_{\max} \leq& \frac{1}{p_0 N},\\
		\delta_{\max} \leq & \frac{1}{p_0}
	\end{align*}
	hold and \eqref{eq:diag_bound2} directly follows, satisfying the conditions for Theorem \ref{t:Main}.
	
	We focus on the graph case $M=1$. We assume that the rate functions $q_{ss'}$ are \emph{affine}, that is, they have the form
	\begin{align}
		\label{eq:linear}
		q_{ss'}\left(\phi\right)=q_{ss'}^{(0)}+\sum_{r \in \S}q_{ss',r}^{(1)}\phi_{r}\rem{,}
	\end{align}
	where $q_{ss'}^{(0)}, \left(q_{ss',r}^{(1)} \right)_{r \in \S}$ are nonnegative constants. Many epidemiological models have this form, including the SIS process.
	
	As it was pointed out in the preliminary work \cite{unified_meanfield}, Szemer\'edi's regularity lemma \cite{Szemeredi} provides a method to approximate \eqref{eq:NIMFA} with a finite system up to arbitrary precision (for large enough $N$).
	
	Roughly speaking, Szemer\'edi's regularity lemma states that any large enough dense graph can be partitioned into finitely many ``boxes'' (called an $\varepsilon$-regular partition) which have the same size (except one remainder box), and besides a few exceptional pairs the edge count between two boxes behaves as if coming from a randomly mixed graph, with error at most $\varepsilon$.
	
	We denote  an $\varepsilon$-regular partition by $V_0,V_1, \dots, V_K$, where $V_0$ is the exceptional set. 
	\begin{align*}
		e(A,B):=\sum_{i \in A} \sum_{j \in B} a_{ij}
	\end{align*}
	refers to the number of edges between the vertex sets $A,B$ with the convention that edges in $A \cap B$ are counted double. 
	
	We define the graph $\overline{\mathcal{G}}$ on vertices $\left(V_1, \dots, V_K\right)$. ($V_0$ is neglected.)   
	
	The adjacency matrix is replaced by the edge density between $A,B \subseteq [N]$ defined as
	\begin{align}
		\label{eq:rho}
		\rho(A,B):=\frac{e\left(A,B\right)}{\left|A\right| \cdot \left|B \right|}
	\end{align}
	It is easy to see that $0 \leq \rho\left(A,B\right) \leq 1.$
	
	The adjacency matrix counterpart for $\overline{\mathcal{G}}$ is simply the edge density between the $V_1,\dots,V_K$ sets. For the average degree we further define
	\begin{align}
		\label{eq:p}
		p:=& \frac{\bar{d}}{N}, \\
		\label{eq:kappa}
		\kappa:=&\frac{\left| V_1\right|}{N}=\dots = \frac{\left| V_K\right|}{N}
	\end{align}
	where $p$ is the global edge density of $\mathcal{G}$ and $\kappa$ is the portion of vertices one box contains. The average degree in $\overline{\mathcal{G}}$ is $K p \approx \frac{p}{\kappa},$ motivating the definition of the weights
	\begin{align}
		\label{eq:w_bar}
		\bar{w}_{kl}:= \frac{\kappa}{p} \rho \left(V_{k},V_{l} \right).
	\end{align}
	
	The corresponding solution of \eqref{eq:NIMFA} on the graph $\overline{\mathcal{G}}$ with weights \eqref{eq:w_bar} is denoted by $\left(v_{k}(t) \right)_{k=1}^{K}$ with initial condition
	\begin{align}
		\label{eq:v(0)} 
		v_{k}(0)=\frac{1}{|V_k|}\sum_{ i \in V_k}z_{i}(0).
	\end{align} 
	
	Finally, we define
	\begin{align}
		\label{eq:v_bar}
		\bar{v}(t):= \sum_{k=1}^{K} \frac{\left|V_k \right|}{N}v_{k}(t)
	\end{align}
	and the average global density vector
	\begin{align}
		\label{eq:z_bar}
		\bar{z}(t):=\frac{1}{N}\sum_{i=1}^{N}z_{i}(t).
	\end{align}
	
	\begin{theorem}
		\label{t:Szemeredi}
		$\forall T>0, \varepsilon>0, p_0>0 \, \exists K_{\max} \in \mathbb{Z}^{+}$ such that for any $\mathcal{G}$ simple graph with density parameter $p_0$ and $N \geq K_{\textrm{max}} $, there exists a partition $V_0, V_1, \dots, V_K$ with $K \leq  K_{\max}$ such that
		\begin{itemize}
			\item $\left|V_1 \right|= \dots =\left|V_K \right|\rem{,}$
			\item $\left|V_0 \right| \leq  \varepsilon N$,
			\item \rem{ $\sup_{0 \leq t \leq T} \left \|\bar{z}(t)-\bar{v}(t) \right \|_1 \leq  \varepsilon.$}
		\end{itemize}
	\end{theorem}
	The proof is provided in Section \ref{s:proofs}.
	
	Szemer\'edi's regularity lemma also guarantees that such a partition can be found in polynomial time \cite{SzemerediAlgorithm}.
	
	We note that $K_{\max}$ may increase rapidly as $\varepsilon \to 0^{+}$  limiting the applicability of the approach. That said, for networks with extra community structure, this approach may still be useful.
	
	\section{Discussion}
	\label{s:Discussion}

	In this paper we examined the accuracy of the so called N-Intertwined Mean Field Approximation on hypergraphs. The idea of NIMFA is to assume vertices are independent from each other, then derive the dynamics of the occupation probabilities of each vertex. This leaves us with and ODE system of size $O(N)$ instead of an exponentially increasing system given by the exact Kolmogorov equations.  
	
	Our findings show that when the incoming weights are well distributed -- for example, vertices typically have large degrees -- then NIMFA gives an accurate approximation. Under additional assumptions we showed how the number of ODEs can be further reduced to give well-known approximation methods from the literature, such as the heterogenous mean field approximation. Finally, we showed how Szemer\'edy's regularity lemma can be used to reduce the number of equations to constant order (depending only on the error desired) for large enough dense graphs.
	
	These results have their limitations. The error bounds work poorly for truly sparse graphs (with bounded average degrees). Analyzing such systems probably requires qualitatively different approaches.
	
	The upper regularity condition can be restrictive for certain applications. We conjecture that the results could be greatly generalized in this direction for degree distributions with fast decaying tails.

	For the reduction for dense graph we applied the strong version of Szemer\'edy's lemma. The weak version of Szemer\'edy's lemma, however, has more desirable algorithmic properties and a smaller bound on the number of "boxes" one needs for a given $\varepsilon$. Extending the theorem in this direction might be beneficial for large, inhomogeneous, dense systems.
	
	Finally, NIMFA has the disadvantage of requiring full knowledge of the network which is usually not possible in practice. Using metapopulation networks instead mitigates this problem, and also greatly reduces the number of equations required. This method, however, relies on the assumption that the metapopulation dynamics is close enough to the original one. Further research is needed to understand how well coarse graining performs in terms of preserving the network dynamics.

	\section{Proofs}
	\label{s:proofs}
	
	\subsection{General proofs}
	
	We state and prove a technical lemma first which will be used throughout other proofs.
	
	\begin{lemma}
		\label{l:technical}
		Let $a_{1},\dots,a_{n}$ and $b_{1},\dots,b_{n}$ two sets of numbers such that $0 \leq \left|a_i \right|, \left|b_i \right| \leq 1$. Then
		\begin{align*}
			\left|\prod_{i=1}^{n}a_{i}-\prod_{i=1}^n b_i \right| \leq \sum_{i=1}^n \left|a_i-b_i \right|.
		\end{align*}
	\end{lemma}
	
	\begin{proof}(Lemma \ref{l:technical})
		
		The proof is by induction on $n$. The statement is trivial for $n=1$. For $n>1$,
		\begin{align*}
			\left|\prod_{i=1}^{n}a_{i}-\prod_{i=1}^n b_i \right| =&\left|a_n\prod_{i=1}^{n-1}a_{i}-b_n\prod_{i=1}^{n-1} b_i \right| =\\
			&\left|\left(a_n -b_n\right)\prod_{i=1}^{n-1}a_{i}+b_n\left(\prod_{i=1}^{n-1}a_{i}-\prod_{i=1}^{n-1} b_i \right) \right| \leq \\
			&\left|a_n -b_n\right|\prod_{i=1}^{n-1}\left|a_{i}\right|+\left|b_n\right| \cdot\left|\prod_{i=1}^{n-1}a_{i}-\prod_{i=1}^{n-1} b_i \right|  \leq \\
			&\left|a_n -b_n\right|+\left|\prod_{i=1}^{n-1}a_{i}-\prod_{i=1}^{n-1} b_i \right|  \leq \left|a_n -b_n\right|+\sum_{i=1}^{n-1}\left|a_i-b_i \right|=\\
			&\sum_{i=1}^{n}\left|a_i-b_i \right|.
		\end{align*}
	\end{proof}
	
	Next we show that \eqref{eq:NIMFA} exhibits a unique global solution.
	\begin{proof}(Theorem \ref{t:NIMFA_Delta})
		
		The right hand side of \eqref{eq:NIMFA} is locally Lipschitz, so there is a unique local solution.
		
		Instead of $q_{ss'}$, we use the modified rate functions	
		\begin{align}
			\label{eq:q_hat}
			\hat{q}_{ss'}(\phi):=& \left|q_{ss'}(\phi) \right|\\
			\nonumber
			\hat{q}_{ss}(\phi)=&-\sum_{s' \neq s}\hat{q}_{s's}(\phi)
		\end{align}	
		which are nonnegative for any input; note that $\left. \hat{q}_{ss'}(\phi) \right|_{\phi \geq 0} =\left. q_{ss'}(\phi) \right|_{\phi \geq 0}.$
		
		The modified version of \eqref{eq:NIMFA} is
		\begin{align*}
			\frac{\d}{\d t} \hat{z}_{i}(t)=\hat{Q} \left( \hat{\zeta}_{i}(t) \right)\hat{z}_i(t)
		\end{align*}
		where $\hat{Q}(\phi)=\left(\hat{q}_{ss'}(\phi) \right)_{s,s' \in \S}.$ The local solution uniquely exist in this case as well, and it either extends to a global solution or blows up at a finite time.
		
		Assume that the local solution blows up at time $t_0$. Then $\hat{\zeta}_{i}(t)$ is well-defined for any $t<t_0.$ 
		
		We construct an auxiliary time-inhomogeneous Markov process on $[0,t_0)$. The state space is $\S$ and the transition  rates at time $t$ are given by the matrix  $\hat{Q}\left( \hat{\zeta}_{i}(t)\right)$.  $p_s(t)$ denotes the probability of being in state $s \in \S.$ The Kolmogorov equations have the form
		\begin{align*}
			\frac{\d }{\d t} p(t)=\hat{Q}\left( \hat{\zeta}_i(t)\right)p(t).
		\end{align*}
		
		Since $\hat{Q} \left(\hat{\zeta}_i(t)\right)$ is continuous for $t<t_0$,
		$$\max_{0 \leq \tau \le t} \left \| \hat{Q} \left(\hat{\zeta}_i(\tau)\right) \right \|$$ exists and is finite.
		
		Based on Grönwall's inequality,
		\begin{align*}
			\hat{z}_{i}(t)-p(t)=& \hat{z}_{i}(0)-p(0) + \int_{0}^{t} \hat{Q}\left( \hat{\zeta}_i(u) \right) \left[\hat{z}_{i}(\tau)-p(\tau) \right]  \d \tau, \\
			\left \|\hat{z}_{i}(t)-p(t)\right \|=& \left \|\hat{z}_{i}(0)-p(0)\right \| +\sup_{0 \leq u \leq t} \left \|\hat{Q}\left( \hat{\zeta}_i(\tau) \right) \right\| \int_{0}^{t} \left \| \hat{z}_{i}(\tau)-p(\tau) \right \|  \d \tau, \\
			\sup_{0 \leq \tau \leq t}	\left \|\hat{z}_{i}(\tau)-p(\tau)\right \| \leq& \left \|\hat{z}_{i}(0)-p(0)\right \| \exp \left(\sup_{0 \leq \tau \leq t} \left \|\hat{Q}\left( \hat{\zeta}_i(\tau) \right) \right\| \cdot t\right).
		\end{align*}
		Choosing $p(0)=\hat{z}_i(0)$ shows that $\hat{z}_{i}(t)=p(t)$ for any $0 \leq t <t_0$ as well.
		
		But $p(t)$ is a probability vector, that is, $\hat{z}_{i}(t) \in \Delta^{\S}$, which contradicts $\hat{z}_{i}(t)$ blowing up as $t\to t_0$, so the solution must be global.
		
		Since the solution is on the simplex $\Delta^{S}$, we have $\hat{q}_{ss'}\left(\hat{\zeta}_i(t) \right)=q_{ss'}\left(\hat{\zeta}_i(t) \right)$ (that is, the absolute values in \eqref{eq:q_hat} are not necessary). Therefore $\hat{z}_i(t)$ is a solution for the original equation \eqref{eq:NIMFA} as well. Since the solution for \eqref{eq:NIMFA} is unique, $\hat{z}_{i}(t)=z_{i}(t).$ This makes $z_i(t)$ a global solution with values on the simplex $\Delta^{S}.$  
	\end{proof}
	
	\subsection{Proof of Theorem \ref{t:Main} \remm{ and \ref{t:fslln}}}
	
	The strategy of the proof is to derive an inequality for $D_{\max}(t)$ and $\tilde{D}_{i}(t)$ such that Grönwall's inequality could be applied.
	
	\remm{
	\begin{lemma}
		\label{c:1}
		Assume $M=1$. Let $\left( \omega_j \right)_{j\in [N]}$ be arbitrary non-negative weights. There exists a constant $\bar{C}=\bar{C}(q_{max})$ such that
		\begin{align*}
			\E \left( \sup_{0 \leq  \tau \leq t} \left | \sum_{j \in [N]} \omega_j \left(\hat{\xi}_{j,s}(\tau)-z_{j,s}(\tau) \right) \right|\right) \leq \bar{C}(1+t)\sqrt{\sum_{j \in [N]} \omega_j^2}.
		\end{align*}
	\end{lemma}
	
	\begin{proof} (Lemma \ref{c:1})
		\begin{align*}
			&\sum_{j \in [N]} \omega_j \left(\hat{\xi}_{j,s}(t)-z_{j,s}(t) \right)=\sum_{j \in [N]} \omega_j \left(\hat{\xi}_{j,s}(0)-z_{j,s}(0) \right) + \\
			&\sum_{s' \neq s}\sum_{j \in [N]}\omega_j \left[ \mathcal{N}_{j,ss'}\left(\mathcal{K}_{j,ss'}(t) \right) -\int_{0}^{t}q_{ss'} \left( \zeta_{j}(\tau)\right) \hat{\xi}_{j,s'}(\tau) \d \tau \right] \\
			&-\sum_{s' \neq s} \sum_{j \in [N]}\omega_j \left[ \mathcal{N}_{j,s's}\left(\mathcal{K}_{j,s's}(t) \right) -\int_{0}^{t}q_{s's} \left( \zeta_{i}(\tau)\right) \hat{\xi}_{j,s}(\tau) \d \tau \right]+ \\
			&\sum_{s'} \int_{0}^{t} \sum_{j \in [N]}\omega_j q_{ss'}(\zeta_j(\tau)) \left(\hat{\xi}_{j,s'}(\tau)-z_{j,s'}(\tau) \right) \d \tau
		\end{align*}
		
		Using the independence of $\hat{\xi}_{j,s}(t)$:
		\begin{align*}
			&\E \left[ \left| \sum_{j \in [N]} \omega_j \left(\hat{\xi}_{j,s}(0)-z_{j,s}(0) \right)\right| \right] \leq \mathbb{D} \left(\sum_{j \in [N]} \omega_j \hat{\xi}_{j,s}(0) \right)= \\
			&\sqrt{\sum_{j \in [N]}\omega_j^2\mathbb{D}^2 \left(\hat{\xi}_{j,s}(0) \right)} \leq \sqrt{\sum_{j \in [N]}\omega_j^2}.
		\end{align*}		
		
		For the remaining terms, we can use the same bound for all $s,s'$  pairs. Notice
		\begin{align*}
			\left|\mathcal{K}_{j,ss'}(t) \right|=\int_{0}^{t}q_{ss'} \left( \zeta_{j}(\tau)\right) \hat{\xi}_{j,s'}(\tau) \d \tau \leq q_{\textrm{max}}t
		\end{align*}
		and the fact that $\mathcal{N}_{j,ss'}\left(\mathcal{K}_{j,ss'}(t) \right) -\int_{0}^{t}q_{ss'} \left( \zeta_{j}(\tau)\right) \hat{\xi}_{j,s'}(\tau) \d \tau $ is a martingale. Using Doob's inequality,
		\begin{align*}
			&\E \left( \sup_{0 \leq \tau \leq t} \left| \sum_{j \in [N]}\omega_j \left[ \mathcal{N}_{j,ss'}\left(\mathcal{K}_{j,ss'}(\tau) \right) -\int_{0}^{\tau}q_{ss'} \left( \zeta_{j}(\tau')\right) \hat{\xi}_{j,s'}(\tau') \d \tau' \right] \right| \right) \leq \\
			&\left(\E\left( \sup_{0 \leq \tau \leq t} \left| \sum_{j \in [N]}\omega_j \left[ \mathcal{N}_{j,ss'}\left(\mathcal{K}_{j,ss'}(\tau) \right) -\int_{0}^{\tau}q_{ss'} \left( \zeta_{j}(\tau')\right) \hat{\xi}_{j,s'}(\tau') \d \tau' \right] \right|\right)^2   \right)^{1/2}\leq\\
			&2\left(\E\left( \left| \sum_{j \in [N]}\omega_j \left[ \mathcal{N}_{j,ss'}\left(\mathcal{K}_{j,ss'}(t) \right) -\int_{0}^{t}q_{ss'} \left( \zeta_{j}(\tau)\right) \hat{\xi}_{j,s'}(\tau) \d \tau \right] \right|\right)^2 \right)^{1/2}=\\
			& 2\left(\E \left\langle\sum_{j \in [N]}\omega_j \left[ \mathcal{N}_{j,ss'}\left(\mathcal{K}_{j,ss'}(\cdot) \right) -\int_{0}^{\cdot}q_{ss'} \left( \zeta_{j}(\tau)\right) \hat{\xi}_{j,s'}(\tau) \d \tau \right] \right\rangle_{\!\! t}\right)^{1/2}=\\
			& 2\left(\E\left\langle\sum_{j \in [N]}\omega_j \mathcal{N}_{j,ss'}\left(\mathcal{K}_{j,ss'}(\cdot) \right) \right\rangle_{\!\! t}\right)^{1/2}=
			2\left(\sum_{j \in [N]}\omega_j^2 \E\left\langle\mathcal{N}_{j,ss'}\left(\mathcal{K}_{j,ss'}(\cdot) \right) \right\rangle_{t}\right)^{1/2}\leq\\
			& 2\left(\sum_{j \in [N]}\omega_j^2 q_{\max}t\right)^{1/2}\leq
			2(1+ q_{\max}t)\sqrt{\sum_{j \in [N]}\omega_j^2},
		\end{align*}
		where $\langle\cdot\rangle_t$ denotes quadratic variation, and we used that the quadratic variation of the integral term is $0$, and the quadratic variation of a pure jump process is the total of the squared jumps (for the Poisson process, jump size is 1, and this reduces to the total number of jumps).
		
		As for the last term,
		\begin{align*}
			&\int_{0}^{t}\E \left[ \left| \sum_{j \in [N]} \omega_j q_{ss'}(\zeta_j(\tau)) \left(\hat{\xi}_{j,s'}(\tau)-z_{j,s'}(\tau) \right) \right| \right] \d \tau \leq \\
			&\int_{0}^{t} \mathbb{D} \left(\sum_{j \in [N]} \omega_j q_{ss'}(\zeta_j(\tau)) \hat{\xi}_{j,s'}(\tau)  \right)\d \tau =\\
			& \int_{0}^{t} \sqrt{\sum_{j \in [N]} \omega_{j}^2 q_{ss'}^2(\zeta_j(\tau)) \mathbb{D}^2 \left( \hat{\xi}_{j,s'}(\tau)\right)} \d \tau \leq q_{\textrm{max}}t \sqrt{\sum_{j \in [N]} \omega_j^2}.
		\end{align*}
	\end{proof}
	}

\remm{
	\begin{proof} (Theorem \ref{t:fslln})
	It is an immediate consequence of Lemma \ref{c:1} as we set the weights to  $\omega_i= \frac{1}{K}\1{i \leq K}.$ 
	\end{proof}
}		
	\remm{Next} we are showing an inequality for the error of the indicators.
	\begin{lemma}
		\label{l:D_0}
		There exists $\tilde{C}_1=\tilde{C}_1(\delta_{\max})$ such that
		\begin{align*}
			D_{\max}^{(0)}(t) \leq& \tilde{C}_1 \int_{0}^{t}D_{\max}(\tau) \d \tau, \\
			\tilde{D}_{i}^{(0)}(t) \leq & \tilde{C}_1 \int_{0}^{t} \tilde{D}_{i}(\tau) \d \tau . 
		\end{align*}
	\end{lemma} 	
	
	\begin{proof} (Lemma \ref{l:D_0})
		$\oplus$ denotes symmetric difference.
		\begin{align*}
			& \left|\xi_{i,s}(\tau)-\hat{\xi}_{i,s}(\tau) \right| \leq \\ 
			&\sum_{\substack{s' \in \S \\ s' \neq s}} \left|\N_{i,ss'} \left(\HH_{i,ss'}(\tau) \right)-\N_{i,ss'} \left(\K_{i,ss'}(\tau) \right) \right|+\left|\N_{i,ss'} \left(\HH_{i,ss'}(\tau) \right)-\N_{i,ss'} \left(\K_{i,ss'}(\tau) \right) \right| \leq \\
			&\sum_{\substack{s' \in \S \\ s' \neq s}} \N_{i,ss'} \left(\HH_{i,ss'}(\tau) \oplus \K_{i,ss'}(\tau)  \right)+\N_{i,s's} \left(\HH_{i,s's}(\tau) \oplus \K_{i,s's}(\tau)  \right) \leq \\
			&\sum_{\substack{s' \in \S \\ s' \neq s}} \N_{i,ss'} \left(\HH_{i,ss'}(t) \oplus \K_{i,ss'}(t)  \right)+\N_{i,s's} \left(\HH_{i,s's}(t) \oplus \K_{i,s's}(t)  \right)
		\end{align*}
		In the last step we used the fact that $\HH_{i,ss'}(\tau) \oplus \K_{i,ss'}(\tau)$ is an increasing set in $\tau$. 
		
		Since the right hand side does not depend on $\tau$, it makes no difference whether we take $\sup_{0 \leq \tau \leq t}$ inside or outside of the expectation.
		\begin{align*}
			&D_{i}^{(0)}(t) \leq \tilde{D}_{i}^{(0)}(t) \leq \\
			& \sum_{s \in \S} \sum_{\substack{s' \in \S \\ s' \neq s}} \E \left[ \N_{i,ss'} \left(\HH_{i,ss'}(t) \oplus \K_{i,ss'}(t)  \right)+\N_{i,s's} \left(\HH_{i,s's}(t) \oplus \K_{i,s's}(t)  \right) \right]
		\end{align*}
		
		The summations with respect to $s$ and $s'$ only contribute a constant factor $\left | \S \right |^2$ which will be neglected. Also, the same bound applies for $E \left[ \N_{i,ss'} \left(\HH_{i,ss'}(t) \oplus \K_{i,ss'}(t)  \right) \right]$ and $E \left[ \N_{i,s's} \left(\HH_{i,s's}(t) \oplus \K_{i,s's}(t)  \right) \right]$, so it is enough to keep track of only the first one, with a factor of $2$.
		
		The rate functions are Lipschitz-continuous on a compact domain due to assumption \eqref{A:1}, so they are bounded; their maximum is denoted by $q_{\max}$.
		
		\begin{align*}
			&\E \left[\N_{i,ss'} \left(\HH_{i,ss'}(t) \oplus \K_{i,ss'}(t)  \right) \right]=\\
			&\E \left[\int_{0}^{t} \left|q_{ss'}\left(\phi_{i}(\tau) \right)\xi_{i,s'}(\tau)-q_{ss'}\left(\remm{\zeta_i(\tau)} \right)\hat{\xi}_{i,s'}(\tau) \right| \d \tau \right]  \leq \\
			& \E \left[ \int_{0}^{t} q_{\max} \left|\xi_{i,s'}(\tau)-\hat{\xi}_{i,s'}(\tau) \right|+L_{q}\sum_{m=1}^{M}\sum_{\underline{r} \in \S^m} \left|\phi_{i,\underline{r}}^{(m)}(\tau)-\remm{\zeta_{i,\underline{r}}^{(m)}(\tau)} \right| \d \tau \right] \leq \\
			& \left(q_{\max}+L_{q}\right)\int_{0}^{t} \sum_{m=0}^{M}D_{i}^{(m)}(\tau) \d \tau \leq \left(q_{\max}+L_{q}\right)\int_{0}^{t} \sum_{m=0}^{M}\tilde{D}_{i}^{(m)}(\tau) \d \tau
		\end{align*}
		
		Setting $\tilde{C}_1:=2\left(q_{\max}+L_q\right)\left| \S \right|^2$ yields
		\begin{align*}
			D_{i}^{(0)}(t) \leq \tilde{D}_{i}^{(0)}(t) \leq \tilde{C}_1 \int_{0}^{t} \sum_{m=0}^{M}D_{i}^{(m)}(\tau) \d \tau \leq \tilde{C}_1 \int_{0}^{t} \underbrace{\sum_{m=0}^{M}\tilde{D}_{i}^{(m)}(\tau)}_{=\tilde{D}_{i}(\tau)} \d \tau.
		\end{align*}  	
	\end{proof}
	
	The second half of the proof of Theorem \ref{t:Main} involves estimating the difference between the neighbors $\phi_{i}(t)$ and $\zeta_{i}(t)$ via the differences of the indicators.
	
	$\zeta_{i}(t)$ does not contain the indicators $\hat{\xi}_{i}(t)$ directly, only their expectation $z_{i}(t)$. To bridge this gap, we introduce ``intermediate neighborhoods''
	\begin{align*}
		\hat{\phi}_{i,\underline{s}}^{(m)}(t)=&\sum_{\underline{j} \in [N]^m}w_{i, \underline{j}}^{(m)} \hat{\xi}_{\underline{j},\underline{s}}^{(m)}(t).
	\end{align*}
	
	Note that under \eqref{NIMFA_initial_condition} and independent initial conditions,
	$$\E \left(\hat{\xi}_{\underline{i},\underline{s}}^{(m)} \right)=\E \left(\prod_{l=1}^{m} \hat{\xi}_{i_l,s_l}(t) \right)=\prod_{l=1}^{m}\E \left( \hat{\xi}_{i_l,s_l}(t) \right)=\prod_{l=1}^{m} z_{i_l,s_l}(t)=z_{\underline{i},\underline{s}}^{(m)}$$
	for non-secondary loop $\underline{i}$ indices. Assumption \eqref{A:2} was made to ensure secondary loops have low total weight. 
	\begin{align}
		\label{eq:diagonal_in_lemma}
		\begin{split}
			&\left|\E \left(\hat{\phi}_{i,\underline{s}}^{(m)}(t) \right)-\zeta_{i,\underline{s}}^{(m)}(t) \right|=\left|\sum_{\underline{j} \in [N]^{m}} w_{i,\underline{j}}^{(m)}\left[\E \left(\hat{\xi}_{\underline{j},\underline{s}}^{(m)}(t)\right)-z_{\underline{j},\underline{s}}^{(m)}(t) \right] \right|=\\
			& \left|\sum_{\substack{\underline{j} \in [N]^{m}\\ \underline{j} \textrm{ s. loop}}} w_{i,\underline{j}}^{(m)}\left[\E \left(\hat{\xi}_{\underline{j},\underline{s}}^{(m)}(t)\right)-z_{\underline{j},\underline{s}}^{(m)}(t) \right] \right| \leq \sum_{\substack{\underline{j} \in [N]^{m}\\ \underline{j} \textrm{ s. loop}}} w_{i,\underline{j}}^{(m)} \leq R \sqrt{\remm{w_{\max}^{*}}}.
		\end{split}
	\end{align}
	
	The next lemma shows that $\hat{\phi}_{i}(t)$ and $\zeta_{i}(t)$ are close.
	\begin{lemma}
		\label{l:intermediate}
		Assume \eqref{NIMFA_initial_condition} holds with independent initial conditions. Then there is a $\tilde{C}_2=\tilde{C}_2\left(\delta_{\max},R\right)$ such that for any $1 \leq m \leq M, \ i \in [N]$
		\begin{align}
			\label{eq:lemma3a}
			\sup_{0 \leq t}\E \left[ \sum_{\underline{s} \in \S^m} \left|\hat{\phi}_{i,\underline{s}}^{(m)}(t)-\zeta_{i,\underline{s}}^{(m)}(t) \right| \right] \leq \tilde{C}_2 \sqrt{\remm{w_{\max}^{*}}}.
		\end{align}
		
		If we further assume $M=1$, there exists a $\tilde{C}_3$  such that for all $t \geq 0$,
		\begin{align}
			\label{eq:lemma3b}
			\E \left[ \sup_{0 \leq t} \sum_{s \in \S} \left|\hat{\phi}_{i,s}(t)-\zeta_{i,s}(t) \right| \right] \leq \tilde{C}_3\remm{(1+t)} \underbrace{\sqrt{\sum_{j=1}^n w_{ij}^2}}_{=\mu_i}.
		\end{align}
	\end{lemma} 
	
	\begin{proof}(Lemma \ref{l:intermediate})
		
		We start by applying \eqref{eq:diagonal_in_lemma}.	
		\begin{align*}
			&\sup_{0 \leq t}\E \left[ \sum_{\underline{s} \in \S^m} \left|\hat{\phi}_{i,\underline{s}}^{(m)}(t)-\zeta_{i,\underline{s}}^{(m)}(t) \right| \right] \leq\\
			& R \left |\S \right|^M \sqrt{\remm{w_{\max}^{*}}}+\sup_{0 \leq t}\E \left[ \sum_{\underline{s} \in \S^m} \left|\hat{\phi}_{i,\underline{s}}^{(m)}(t)-\E \left(\hat{\phi}_{i,\underline{s}}^{(m)}(t) \right) \right| \right].
		\end{align*}
		
		The first term is of the desired form; we examine the second term.
		\begin{align*}
			&\E \left[ \sum_{\underline{s} \in \S^m} \left|\hat{\phi}_{i,\underline{s}}^{(m)}(t)-\E \left(\hat{\phi}_{i,\underline{s}}^{(m)}(t) \right) \right| \right]=\sum_{\underline{s} \in \S^m}\E \left( \left|\hat{\phi}_{i,\underline{s}}^{(m)}(t)-\E \left(\hat{\phi}_{i,\underline{s}}^{(m)}(t) \right) \right| \right) \leq \\
			&\sum_{\underline{s} \in \S^m} \sqrt{\mathbb{D}^2 \left( \hat{\phi}_{i,\underline{s}}^{(m)}(t) \right)}=\remm{\sum_{\underline{s} \in \S^m} \sqrt{\sum_{\underline{j}, \ul{k} \in [N]^m} w_{i, \underline{j}}^{(m)}w_{i, \underline{k}}^{(m)} \cov \left( \hat{\xi}_{\underline{j},\underline{s}}^{(m)}(t),\hat{\xi}_{\underline{k},\underline{s}}^{(m)}(t) \right)} \leq} \\
			& \remm{\left |\S \right|^M \sqrt{\sum_{\substack{\underline{j}, \ul{k} \in [N]^m \\ \ul{j} \cap \ul{k} \neq \emptyset}} w_{i, \underline{j}}^{(m)}w_{i, \underline{k}}^{(m)}  } \leq \left |\S \right|^M \sqrt{w_{\max}^{*}}}. 
		\end{align*}
		The bound is uniform in $t$, so it can be upgraded to $\sup_{0\leq t}$ for free, and \eqref{eq:lemma3a} holds with $\tilde{C}_2=\remm{\left(R+1\right)}\left|\S \right|^M.$
		
		\remm{ \eqref{eq:lemma3b} is a consequence of Lemma \ref{c:1} by setting the weights to $\omega_{j}=w_{ij}$.}
	\end{proof}	
	
	Next we show an upper bound for the differences of neighborhood vectors, which are captured by the values $D^{(m)}_{\max}(t)$.
	\begin{lemma}
		\label{l:environment}
		Assume \eqref{NIMFA_initial_condition} and independent initial conditions. Then there exist constants $ \tilde{C}_{4}=\tilde{C}_{5}\left(\delta_{\max} \right)$ such that for any $t \geq 0$ and $1 \leq m \leq M$
		\begin{align*}
			D^{(m)}_{\max}(t) \leq \tilde{C}_{2}\sqrt{\remm{w_{\max}^{*}}}+\tilde{C}_4 D_{\max}^{(0)}(t).
		\end{align*}
		where $\tilde{C}_2$ comes from Lemma \ref{l:intermediate}.
		
		If we further assume $M=1$ then
		\begin{align*}
			\tilde{D}^{(1)}(t) \leq \tilde{C}_3\remm{(1+t)} \mu+W \tilde{D}^{(0)}(t).
		\end{align*}
		where $\tilde{C}_3$ comes from Lemma \ref{l:intermediate}.
	\end{lemma}
	
	\begin{proof}(Lemma \ref{l:environment})
		
		Using Lemma \ref{l:intermediate}, we have
		\begin{align*}
			&D_{i}^{(m)}(t)=\sup_{0 \leq \tau \leq t}\E \left[ \sum_{\underline{s} \in \S^m} \left|\phi_{i, \underline{s}}^{(m)}(\tau)-\zeta_{i,\underline{s}}^{(m)}(\tau) \right| \right] \leq\\
			& \tilde{C}_2 \sqrt{w_{\max}}+\sup_{0 \leq \tau \leq t}\E \left[ \sum_{\underline{s} \in \S^m} \left|\phi_{i, \underline{s}}^{(m)}(\tau)-\hat{\phi}_{i,\underline{s}}^{(m)}(\tau) \right| \right] \leq \\
			&\tilde{C}_2 \sqrt{w_{\max}}+\sum_{\underline{j} \in [N]^m}w_{i,\underline{j}}^{(m)}\left(\sup_{0 \leq \tau \leq t}\E \left[ \sum_{\underline{s} \in \S^m} \left|\xi_{\underline{j}, \underline{s}}^{(m)}(\tau)-\hat{\xi}_{\underline{j},\underline{s}}^{(m)}(\tau) \right| \right] \right).
		\end{align*}
		
		Lemma \ref{l:technical} provides
		\begin{align*}
			&\left|\xi_{\underline{j}, \underline{s}}^{(m)}(\tau)-\hat{\xi}_{\underline{j},\underline{s}}^{(m)}(\tau) \right| \leq \sum_{l=1}^{m} \left|\xi_{j_l,s_l}(\tau)-\hat{\xi}_{j_l,s_l}(\tau) \right| \\
			&\sup_{0 \leq \tau \leq t}\E \left[ \sum_{\underline{s} \in \S^m} \left|\xi_{\underline{j}, \underline{s}}^{(m)}(\tau)-\hat{\xi}_{\underline{j},\underline{s}}^{(m)}(\tau) \right| \right] \leq \sup_{0 \leq \tau \leq t}\E \left[ \sum_{\underline{s} \in \S^m} \sum_{l=1}^{m} \left|\xi_{j_l,s_l}(\tau)-\hat{\xi}_{j_l,s_l}(\tau) \right| \right] \leq\\
			& \left|S \right|^M \sum_{l=1}^{m} \sup_{0 \leq \tau \leq t } \E \left[\sum_{r \in \S} \left|\xi_{j_l,r}(\tau)-\hat{\xi}_{j_l,r}(\tau) \right| \right] \leq \left|S \right|^M \sum_{l=1}^{m} D_{j_l}^{(0)}(t) \leq M \left|\S \right|^M D_{\max}^{(0)}(t).
		\end{align*}
		
		Putting the inequalities together yields
		\begin{align*}
			D_{i}^{(m)}(t) \leq& \tilde{C}_2 \sqrt{w_{\max}}+M \left|\S \right|^M D_{\max}^{(0)}(t) \underbrace{\sum_{\underline{j} \in [N]^m}w_{i,\underline{j}}^{(m)}}_{=\delta^{(m)}(i) } \\
			D_{\max}^{(m)}(t) \leq& \tilde{C}_2 \sqrt{w_{\max}}+\underbrace{M \left|\S \right|^M \delta_{\max}}_{=:\tilde{C}_4} D_{\max}^{(0)}(t).
		\end{align*}
		
		For the second part of Lemma \ref{l:environment}, we once again use Lemma \ref{l:intermediate}.
		\begin{align*}
			&\tilde{D}_{i}^{(1)}(t)= \E \left[\sup_{0 \leq \tau \leq t} \sum_{s \in \S}  \left| \phi_{i,s}(\tau)-\zeta_{i,s}(\tau) \right| \right] \leq \\
			& \tilde{C}_3\remm{(1+t)}\mu_{i}+ \E \left[\sup_{0 \leq \tau \leq t} \sum_{s \in \S}  \left| \phi_{i,s}(\tau)-\hat{\phi}_{i,s}(\tau) \right| \right] \leq\\
			& \tilde{C}_3\remm{(1+t)}\mu_{i}+\sum_{j=1}^{N} w_{ij} \left( \E \left[\sup_{0 \leq \tau \leq t} \sum_{s \in \S}  \left| \xi_{j,s}(\tau)-\hat{\xi}_{j,s}(\tau) \right| \right] \right)=\tilde{C}_3\remm{(1+t)} \mu_i+\sum_{j=1}^N w_{ij}\tilde{D}_{j}^{(0)}(t),
		\end{align*}
		so
		\begin{align*}
			\tilde{D}^{(1)}(t) \leq \tilde{C}_3\remm{(1+t)} \mu+W \tilde{D}^{(0)}(t).
		\end{align*}
	\end{proof}	
	
	With all the preparations done, we finally turn to proving Theorem \ref{t:Main}.
	\begin{proof}(Theorem \ref{t:Main})
		
		Using Lemma \ref{l:D_0} and \ref{l:environment} and Grönwall's inequality yields
		\begin{align*}
			&D_{\max}(t)=D_{\max}^{0}(t)+\sum_{m=1}^{M} D_{\max}^{(m)}(t) \leq\\
			& M\tilde{C}_2 \sqrt{\remm{w_{\max}^{*}}}+\left(M \tilde{C}_4+1 \right)D_{\max}^{(0)}(t) \leq\\
			& M\tilde{C}_2 \sqrt{\remm{w_{\max}^{*}}}+\left(M \tilde{C}_4+1 \right)\int_{0}^{t}D_{\max}(\tau) \d \tau,
		\end{align*}
		so
		\begin{align*}
			D_{\max}(t) \leq  \underbrace{M \tilde{C}_2e^{\left(M \tilde{C}_4+1 \right)t}}_{=:C} \sqrt{\remm{w_{\max}^{*}}}. 
		\end{align*}
		
		Proving the second part is similar.
		\begin{align*}
			\tilde{D}(t)=&\tilde{D}^{(0)}(t)+\sum_{m=1}^{M}\tilde{D}^{(m)}(t) \leq \underbrace{M \tilde{C}_3}_{=:C_1}\remm{(1+t)}\mu +M \left(W+I\right)\tilde{D}^{(0)}(t) \leq \\
			&C_1 \remm{(1+t)} \mu+ \underbrace{\tilde{C_1}M}_{=:C_2} \int_{0}^{t}\left(W+I\right)\tilde{D}(\tau) \d \tau \Rightarrow \\
			\left \|\tilde{D} (t)\right \| \leq & C_1 \left \| \mu \right \|+C_2 \left\|W+I \right\|\int_{0}^{t} \left \|\tilde{D}(\tau) \right \| \d \tau,
		\end{align*}
		so
		\begin{align*}
			\left \|\tilde{D} (t)\right \| \leq & C_1\remm{(1+t)}e^{C_2 \left\|W+I \right\| t} \left \| \mu \right \|.
		\end{align*}
	\end{proof}
	
	\subsection{Proof of Theorem  \ref{t:xi_z}}

	\begin{proof}(Theorem \ref{t:xi_z})
		For \eqref{eq:thm4a}, we consider $0\leq \tau\leq t$ and use both Theorems \ref{t:Main} and \ref{t:fslln}:
		\begin{align*}
			&\E \left(\sum_{s \in \S} \left|\frac1N\sum_{i=1}^N\left(\xi_{i,s}(\tau)-z_{i,s}(\tau)\right) \right| \right)\leq\\
			&\quad\sum_{s \in \S}\E \left(\left|\frac1N\sum_{i=1}^N\left(\xi_{i,s}(\tau)-\hat\xi_{i,s}(\tau)\right) \right| +\left|\frac1N\sum_{i=1}^N\left(\hat\xi_{i,s}(\tau)-z_{i,s}(\tau)\right) \right|\right)\leq\\
			&\quad\frac1N\sum_{i=1}^N\underbrace{\sum_{s \in \S}\E \left|\left(\xi_{i,s}(\tau)-\hat\xi_{i,s}(\tau)\right) \right|}_{\leq D_{\max}(t)} +\sum_{s \in \S}
			\underbrace{\E \left|\frac1N\sum_{i=1}^N\left(\hat\xi_{i,s}(\tau)-z_{i,s}(\tau)\right) \right|}_{\leq 2/\sqrt{N}}\leq\\
			&\quad D_{\max}(t)+\frac{2|\S|}{\sqrt{N}}\leq C\left(\sqrt{w_{\max}}+\frac{1}{\sqrt{N}}\right).
		\end{align*}
		
		The derivation of \eqref{eq:thm4b} is analogous to \eqref{eq:thm4a} with the exception of keeping the $\sup_{0 \leq \tau \leq t}$ inside the expectation and using \eqref{eq:main2} instead of \eqref{eq:main1}.
		
		For \eqref{eq:thm4c}, we just note that
		\begin{align*}
			\E\left[\sup_{0 \leq \tau \leq t}  \left(\sum_{s \in \S} \left|\frac1N\sum_{i=1}^N\left(\xi_{i,s}(\tau)-\hat\xi_{i,s}(\tau)\right) \right| \right)\right] &\leq
			\frac1N\|\tilde{D}(t)\|_1\\
			&\leq \frac1{\sqrt{N}}\| \tilde{D}(t)\|_2=O\left(\sqrt{\frac{1}{N}\|\mu \|_2^2} \right),
		\end{align*}
		and the rest of the argument is essentially identical to the previous one.
		
	\end{proof}
	
	\subsection{Proof of Proposition \ref{p:activity}}
	
	Let $p_{k^m}^{(m)}:=\frac{N_{k^{(m)}}}{N}$ denote the ratio of vertices in the local group $k^m.$
	
	\begin{align}
		\nonumber
		\bar{\zeta}_{k}^{(m)}(t)&=\sum_{\underline{l}^{(m)}} \bar{w}_{k^{(m)},\underline{j}^{(m)}}^{(m)}\prod_{r=1}^{m}\E \left(\left. z_{\iota}(t) \right| k^{(m)}(\iota)=l_r^{(m)} \right)\\
		\nonumber
		&= \sum_{\underline{l}^{(m)}}\left( \prod_{r=1}^m p_{l_r}^{(m)}\right)\left(a_{k^m}^{(m)}+\sum_{r=1}^{m}a_{l_r^{(m)}}^{(m)} \right)\prod_{r=1}^{m}\E \left(\left. z_{\iota}(t) \right| k^{(m)}(\iota)=l_r^{(m)} \right)\\
		\label{eq:activity_calc}
		&= \sum_{\underline{l}^{(m)}}\left(a_{k^m}^{(m)}+\sum_{r=1}^{m}a_{l_r^{(m)}}^{(m)} \right)\prod_{r=1}^{m}p_{l_r}^{(m)}\E \left(\left. z_{\iota}(t) \right| k^{(m)}(\iota)=l_r^{(m)} \right)
	\end{align}
	
	Observe
	\begin{align*}
		&\sum_{l^{m}=1}^{K^{(m)}}p_{l_r}^{(m)}\E \left(\left. z_{\iota}(t) \right| a^{(m)}(\iota)=l_r^{(m)} \right)=\\
		&\E \left(\E \left(\left. z_{\iota}(t) \right| a^{(m)}(\iota)=l_r^{(m)} \right) \right)=\E \left( z_{\iota}(t)  \right).
	\end{align*}
	Also introduce
	\begin{align*}
		\psi^{(m)}(t):= \sum_{l=1}^{K^{(m)}}a_{l}^{(m)}p_{l}^{(m)}\E \left(\left. z_{\iota}(t) \right| a^{(m)}(\iota)=l^{(m)} \right)
	\end{align*}
	which is renaissance of an activity biased average.
	
	We expand \eqref{eq:activity_calc} based on the terms $a_{k^m}^{(m)}+\sum_{r=1}^{m}a_{l_r^{(m)}}^{(m)}$. For $a_{k^m}^{(m)}$
	\begin{align*}
		&a_{k^m}^{(m)}\sum_{\underline{l}^{(m)}}\prod_{r=1}^{m}p_{l_r}^{(m)}\E \left(\left. z_{\iota}(t) \right| k^{(m)}(\iota)=l_r^{(m)} \right)=\\
		& a_{k^m}^{(m)} \left(\sum_{l=1}^{K^(m)}p_{l}^{(m)}\E \left(\left. z_{\iota}(t) \right| k^{(m)}(\iota)=l^{(m)} \right) \right)^m=\\
		&a_{k^m}^{(m)}\E^{m} \left(z_{\iota}(t)\right) .
	\end{align*}
	For the $a_{l_r'}^{(m)}$ terms we have
	\begin{align*}
		&\sum_{\underline{l}^{(m)}}a_{l_{r'}^{(m)}}\prod_{r=1}^{m}p_{l_r}^{(m)}\E \left(\left. z_{\iota}(t) \right| k^{(m)}(\iota)=l_r^{(m)} \right)=\\
		&\underbrace{\sum_{l_{r'}=1}^{K^{(m)}}a_{l_{r'}^{(m)}}p_{l_{r'}}^{(m)}\E \left(\left. z_{\iota}(t) \right| k^{(m)}(\iota)=l_{r'}^{(m)} \right)}_{\psi^{(m)}(t)}\sum_{\substack{l_r^{(m)}=1 \\  r \neq r'}}^{K^{(m)}}\prod_{\substack{r =1 \\ r \neq r'}}^{m}p_{l_r}^{(m)}\E \left(\left. z_{\iota}(t) \right| k^{(m)}(\iota)=l_r^{(m)} \right)=\\
		& \psi^{(m)}(t) \left( \sum_{l=1}^{K^(m)}p_{l}^{(m)}\E \left(\left. z_{\iota}(t) \right| k^{(m)}(\iota)=l^{(m)} \right) \right)^{m-1}=\psi^{(m)}(t) \E^{m-1} \left(z_{\iota}(t) \right).
	\end{align*}
	
	Therefore, \eqref{eq:activity_calc} reduces to
	\begin{align*}
		\bar{\zeta}^{(m)}_k(t)&=a_{k^m}^{(m)}\E^{m} \left(z_{\iota}(t)\right)+\psi^{(m)}(t) \E^{m-1} \left(z_{\iota}(t) \right)\\
		&=\left(a_{k^m}^{(m)}\E \left(z_{\iota}(t)\right)+\psi^{(m)}(t) \right)\E^{m-1} \left(z_{\iota}(t)\right).
	\end{align*}	
	
	\subsection{Proof of Theorem \ref{t:Szemeredi}}
	
	Recall \eqref{eq:rho}. We call the sets $X,Y \subset [N]$ $\varepsilon$-regular if for all $A \subseteq X, \ B \subseteq Y$ such that $\left|A \right|>\varepsilon \left|X \right|, \ \left| B\right|>\varepsilon \left|Y \right| $ one has
	\begin{align*}
		\left|\rho\left(A,B \right)-\rho \left( X,Y\right) \right|<\varepsilon.
	\end{align*} 
	
	We use Szemer\'edi's regularity lemma.
	\begin{lemma*}(Szemer\'edi's regularity lemma)
		
		For every $\varepsilon>0, \ K_{\min} \in \mathbb{Z}^{+}$ there is a $K_{\max}$ such that if $N\geq K_{\max}$ there is a partition $V_{0}, V_{1}, \dots, V_{K}$ such that
		\begin{align*}
			&\left|V_0 \right| <\varepsilon N,\\
			& \left|V_1 \right|= \dots =\left| V_K \right|, \\
			& K_{\min} \leq K \leq K_{\max}
		\end{align*}
		and there are at most $\varepsilon {K \choose 2 }$ pairs of $\left(V_{k},V_{l}\right), \ 1 \leq k <l \leq K$ such that they are not $\varepsilon$-regular.
	\end{lemma*}	
	
	Fix a $\varepsilon'>0$ and a $K_{\min}$ such that
	\begin{align*}
		K_{\min}>\frac{1}{\varepsilon'}.
	\end{align*}
	This choice ensures that there are enough boxes such that most of the vertices are between boxes and not within them. This is a fairly common approach in the context of Szemer\'edi's regularity lemma \cite{Szemeredi}.
	
	Using Szemer\'edi's regularity lemma for $\varepsilon'$, we obtain a partition denoted by $V_0,V_1, \dots, V_K.$
	
	For $p$ and $\kappa$, as defined in \eqref{eq:p} and \eqref{eq:kappa}, we have the following inequalities:
	\begin{align*}
		p=\frac{\bar{d}}{N} \geq \frac{p_0(N-1)}{N} \geq \frac{p_0}{2}>0\\
		1= \sum_{k=0}^{K}\frac{\left|V_k \right|}{N} \geq \sum_{k=1}^{K}\frac{\left|V_k \right|}{N}=K \kappa \Longrightarrow\kappa \leq  \frac{1}{K} \leq \frac{1}{K_{\min}}<\varepsilon'
	\end{align*}
	where we used $N\geq 2$.
	
	Introduce the notations
	\begin{align*}
		\bar{z}_{k}(t):=&\frac{1}{\left|V_k \right|}\sum_{i \in V_k} z_{i}(t),\\
		\psi(t):=&\sum_{k=1}^{K} \frac{\left|V_k \right|}{N}\left \|\bar{z}_{k}(t)-v_k(t) \right \|_1=\kappa \sum_{k=1}^{K} \left \|\bar{z}_{k}(t)-v_k(t) \right \|_1.
	\end{align*}
	If $V_0=\emptyset,$ we use the convention $z_{0}(t) \equiv 0.$
	
	From \eqref{eq:z_bar} and \eqref{eq:v_bar}, we have
	\begin{align*}
		&\bar{z}(t)=\frac{1}{N}\sum_{i=1}^{N}z_{i}(t)=\sum_{k=0}^{K} \frac{\left|V_k \right|}{N}\frac{1}{\left| V_k \right|}\sum_{i \in V_k}z_{i}(t)=\sum_{k=0}^{K} \frac{\left|V_k \right|}{N} \bar{z}_k(t) \\	
		&\left\|\bar{z}(t)-\bar{v}(t) \right \|_1= \left \| \frac{\left|V_0 \right|}{N} \bar{z}_{0}(t)+\sum_{k=1}^{K} \frac{\left|V_k \right|}{N} \left[\bar{z}_k(t)-v_k(t) \right] \right \|_1 \leq \\
		& \frac{\left|V_0 \right|}{N} \left \|\bar{z}_{0}(t) \right \|_{1}+\sum_{k=1}^{K} \frac{\left|V_k \right|}{N} \left \|\bar{z}_k(t)-v_k(t) \right \|_1 \leq \varepsilon'+\psi(t)
	\end{align*}
	where in the last step we used $\left|V_0 \right|<\varepsilon'N$ and
	\begin{align*}
		&  \left \|\bar{z}_{0}(t) \right \|_{1} \leq \frac{1}{\left|V_0 \right |}\sum_{i \in V_0}\underbrace{ \left \|z_{i}(t) \right \|_1}_{=1}=1.	
	\end{align*}
	Going forward, it is enough to examine $\psi(t).$
	
	Next we calculate the derivative of $\bar{z}_{k}(t).$ As $M=1,$ \eqref{eq:HMFA} takes the form
	\begin{align*}
		\frac{\d}{\d t} z_{i,s}(t)=&\sum_{s' \in \mathcal{S}}q_{ss'} \left(\zeta_{i}(t)\right)z_{i,s'}(t)=\\
		&\sum_{s' \in \mathcal{S}}q_{ss'}^{(0)} z_{i,s'}(t)+\sum_{s' \in \mathcal{S}}\sum_{r \in \mathcal{S}}q_{ss',r}^{(1)} \zeta_{i,r}(t)z_{i,s'}(t)=\\
		&\sum_{s' \in \mathcal{S}}q_{ss'}^{(0)} z_{i,s'}(t)+\sum_{s' \in \mathcal{S}}\sum_{r \in \mathcal{S}}q_{ss',r}^{(1)} \left[\sum_{j=1}^{N} \underbrace{\frac{a_{ij}}{\bar{d}}}_{w_{ij}^{(m)}}z_{i,s'}(t)z_{j,r}(t) \right]\\
		\frac{\d}{\d t}\bar{z}_{k,s}(t)=& \sum_{s' \in \mathcal{S}}q_{ss'}^{(0)} \bar{z}_{k,s'}(t)+\sum_{s' \in \mathcal{S}}\sum_{r \in \mathcal{S}}q_{ss',r}^{(1)} \left[\frac{1}{\left|V_k \right|}\sum_{i \in V_k}\sum_{j=1}^{N} \frac{a_{ij}}{\bar{d}}z_{i,s'}(t)z_{j,r}(t) \right]
	\end{align*}
	Similarly,
	\begin{align*}
		\frac{\d}{\d t} v_{k,s}(t)=&\underbrace{\sum_{s' \in \mathcal{S}}q_{ss'}^{(0)} v_{k,s'}(t)+\sum_{s' \in \mathcal{S}}\sum_{r \in \mathcal{S}}q_{ss',r}^{(1)}\sum_{l=1}^{K}{\bar{w}_{kl}}v_{k,s'}(t)v_{l,r}(t)}_{=:f_{k,s}\left(V(t)\right)}
	\end{align*}
	where $V(t):=\left(v_{k,s}(t)\right)_{k \in [K], \ s \in \mathcal{S}}$ and $\overline{Z}(t)=\left(\bar{z}_{k,s}(t)\right)_{k \in [K], \ s \in \mathcal{S}}$ analogously.
	
	Next we show a Lipschitz-type inequality for $f_{k}=\left(f_{k,s} \right)_{s \in \mathcal{S}}.$ 
	\begin{align*}
		&\left|\bar{z}_{k,s'}(t)\sum_{l=1}^{K}\bar{w}_{kl}\bar{z}_{l,r}(t)	-v_{k,s'}(t)\sum_{l=1}^{K}\bar{w}_{kl}v_{l,r}(t) \right| \leq \\
		&\left|\bar{z}_{k,s'}(t)-v_{k,s'}(t) \right| \underbrace{\sum_{l=1}^{K}\bar{w}_{kl}\bar{z}_{l,r}(t)}_{\leq \sum_{k=1}^K\bar{w}_{kl} \leq \frac{2}{p_0 K}K=\frac{2}{p_0}}+\underbrace{v_{k,s'}(t)}_{\leq 1}\sum_{l=1}^{K}\underbrace{\bar{w}_{kl}}_{\leq \frac{2\kappa}{p_0} }\left|\bar{z}_{l,r}(t)-v_{l,r}(t)\right| \leq \\
		&\frac{2}{p_0}\left( \left|\bar{z}_{k,s'}(t)-v_{k,s'}(t) \right|+\kappa \sum_{l=1}^{K}\left|\bar{z}_{l,r}(t)-v_{l,r}(t) \right|\right),
	\end{align*}
	so
	\begin{align*} &\left|f_{k,s}\left(\bar{Z}(t) \right)-f_{k,s}\left(V(t) \right) \right|\leq 	
		q_{\max}\sum_{s' \in \mathcal{S}} \left|\bar{z}_{k,s'}(t)-v_{k,s'}(t) \right|+\\
		&\frac{2q_{\max}}{p_0}\sum_{s' \in \mathcal{S}}\sum_{r \in \mathcal{S}} \left( \left|\bar{z}_{k,s'}(t)-v_{k,s'}(t) \right|+\kappa \sum_{l=1}^{K}\left|\bar{z}_{l,r}(t)-v_{l,r}(t) \right|\right) = \\
		&q_{\max}\left(1+\frac{2 \left|\mathcal{S} \right|}{p_0} \right)\left \|\bar{z}_{k}(t)-v_{k}(t) \right \|_1+\frac{2q_{\max}\left|\mathcal{S} \right|}{p_0} \psi(t).
	\end{align*}
	Summation for $s\in \mathcal{S}$ results only in an extra $\mathcal{S}$ factor, so there exists a constant $L_f$ such that
	\begin{align}
		\label{eq:L_f}
		\left \|f_k \left(\overline{Z}(t) \right)-f_k \left(V(t)\right) \right \|_1 \leq L_f \left(\left \|\bar{z}_k(t)-v_k(t) \right \|_1+\psi(t) \right).
	\end{align}
	
	Next we look to replace the right hand side of $\frac{\d}{\d t} \bar{z}_{k,s}(t)$ with $f_{k,s}\left(\overline{Z}(t)\right).$ The corresponding error term is
	\begin{align}
		\label{eq:g}
		g_{k,s}(t):=\sum_{s' \in \mathcal{S}}\sum_{r \in \mathcal{S}}q_{ss',r}^{(1)} \left[\frac{1}{\left|V_k \right|}\sum_{i \in V_k}\sum_{j=1}^{N} \frac{a_{ij}}{\bar{d}}z_{i,s'}(t)z_{j,r}(t)-\sum_{l=1}^{K}{\bar{w}_{kl}}\bar{z}_{k,s'}(t)\bar{z}_{l,r}(t) \right],
	\end{align}
	and from $\frac{\d}{\d t} \bar{z}_{k}(t)=g_{k}(t)+f_k \left(\overline{Z}(t)\right)$, we have
	\begin{align*}
		\bar{z}_{k}(t)=& \bar{z}_{k}(0)+\int_{0}^{t}g_k(\tau)\d \tau+\int_{0}^{t}f_{k}\left(\overline{Z}(\tau)\right) \d \tau.
	\end{align*}
	
	Using $\bar{z}_k(0)=v_k(0)$, $\psi(t)$ can be bounded from above by
	\begin{align*}
		\psi(t)=&\kappa \sum_{k=1}^{K} \left \|\bar{z}_k(t) -v_k(t)\right \|_1 \leq\\
		&t \cdot \sup_{0 \leq \tau \leq t} \kappa \sum_{k=1}^{K} \left \|g_{k}(\tau) \right \|_1+\int_{0}^{t} \kappa \sum_{k=1}^{K} \left \|f_{k}\left(\overline{Z}(\tau) \right)-f_k \left(V(\tau)\right) \right \|_1 \d \tau \leq \\
		& t \cdot \sup_{0 \leq \tau \leq t} \kappa \sum_{k=1}^{K} \left \|g_{k}(\tau) \right \|_1+L_f\int_{0}^{t} \kappa \sum_{k=1}^{K} \left(\left \|\bar{z_k}(\tau)-v_k(\tau)\right \|_1+\psi(\tau) \right) \d \tau \leq \\
		&  t \cdot \sup_{0 \leq \tau \leq t} \kappa \sum_{k=1}^{K} \left \|g_{k}(\tau) \right \|_1+2L_{f} \int_{0}^{t}\psi(\tau) \d \tau,
	\end{align*}
	so from Grönwall's inequality,
	\begin{align*}
		\sup_{0 \leq t \leq T} \psi(t) \leq & \left(T \cdot \sup_{0 \leq t \leq T} \kappa \sum_{k=1}^{K} \left \|g_{k}(t) \right \|_1 \right) e^{2L_f T}.
	\end{align*}
	
	Therefore it is enough to show that $\sup_{0 \leq t \leq T} \kappa \sum_{k=1}^{K} \left \|g_{k}(t) \right \|_1 =O\left(\varepsilon'\right)$, and with an appropriate choice of $\varepsilon=C\varepsilon'$ we can conclude
	\begin{align*}
		&\sup_{0 \leq t \leq T} \left \|\bar{z}(t)-\bar{v}(t) \right \|_1 \leq  \varepsilon.
	\end{align*}
	
	\begin{align*}
		&\kappa \sum_{l=1}^K \left \|g_{k}(t) \right \|_{1} =\\
		& \kappa \sum_{s \in \mathcal{S}}\sum_{k=1}^K \left| \sum_{s',r \in \mathcal{S}}q_{ss',r}^{(1)} \left[\frac{1}{\left|V_k \right|}\sum_{i \in V_k}\sum_{j=1}^{N} \frac{a_{ij}}{\bar{d}}z_{i,s'}(t)z_{j,r}(t)-\sum_{l=1}^{K}{\bar{w}_{kl}}\bar{z}_{k,s'}(t)\bar{z}_{l,r}(t) \right] \right| \leq \\
		& \kappa q_{\max} \sum_{s,s',r\in \mathcal{S}}\sum_{k=1}^{K}\sum_{l=0}^{K} \left| \frac{1}{|V_k|}\sum_{i \in V_k}\sum_{j \in V_l} \frac{a_{ij}}{\bar{d}}z_{i,s'}(t)z_{j,r}(t)-\bar{w}_{kl}\bar{z}_{k,s'}(t)\bar{z}_{l,r}(t) \right|	
	\end{align*}
	
	$\sum_{s,s',r\in \mathcal{S}}(\dots)$ only contributes a factor of $\left |\mathcal{S} \right |^3$ which we can include in the constant factor along with $q_{\max}.$ The remaining terms are
	\begin{align}
		\label{eq:hard_work}
		\kappa \sum_{k=1}^{K}\sum_{l=0}^{K} \left| \frac{1}{|V_k|}\sum_{i \in V_k}\sum_{j \in V_l} \frac{a_{ij}}{\bar{d}}z_{i,s'}(t)z_{j,r}(t)-\bar{w}_{kl}\bar{z}_{k,s'}(t)\bar{z}_{l,r}(t) \right|.	
	\end{align}
	
	In the next step we shall get rid of the diagonal $(k,l)$ terms and also the terms with $l=0$. We have
	\begin{align*}
		& \frac{1}{|V_k|}\sum_{i \in V_k}\sum_{j \in V_l} \frac{a_{ij}}{\bar{d}}z_{i,s'}(t)z_{j,r}(t) \leq \frac{1}{\left|V_k \right| \bar{d}}\sum_{i \in V_k}\sum_{j \in V_k} 1=\frac{\left |V_l \right |}{\bar{d}} \leq \frac{\varepsilon'}{p} \leq \frac{2 \varepsilon'}{p_0},\\
		& \bar{w}_{kl}\bar{z}_{k,s'}(t)\bar{z}_{l,r}(t) \leq \frac{\kappa}{p} \leq \frac{2\varepsilon'}{p_0},
	\end{align*}
	so each term in the sum of \eqref{eq:hard_work} is $O \left(\varepsilon' \right).$ There are $O(K)$ pairs which are either diagonal or $l=0$, so their overall contribution to the sum is $O\left(\kappa K \varepsilon' \right)=O\left(\varepsilon' \right),$ hence we can neglect them and what we are left with is
	\begin{align}
		\label{eq:hard_work2}
		\kappa\ \sum_{(k,l) \in \mathcal{I}} \left| \frac{1}{|V_k|}\sum_{i \in V_k}\sum_{j \in V_l} \frac{a_{ij}}{\bar{d}}z_{i,s'}(t)z_{j,r}(t)-\bar{w}_{kl}\bar{z}_{k,s'}(t)\bar{z}_{l,r}(t) \right|.	
	\end{align}
	where $\mathcal{I}=\{(k,l) | k,l \in [K], k \neq l \}$.
	
	In order to have an upper bound for \eqref{eq:hard_work2} we want to use the properties of the $\varepsilon'$-regular partition. However, Szemer\'edi's regularity lemma uses subsets of $[N]$, or in other words, $0-1$ valued indicators of vertices compared to $z_{i,s}(t)$ which may take any value from $[0,1].$
	
	To account for this problem, we introduce $N$ independent homogeneous Markov processes taking values from $\mathcal{S}$. Each process makes Markov transitions according to the transition rate matrix $Q \left(\zeta_{i}(t)\right)$ and its initial distribution is given by $\left(z_{i,s}(0)\right)_{s \in \mathcal{S}}.$ Let $\eta_{i,s}(t)$ be an indicator of the $i$'th such process is at state $s$ at time $t$. We also apply the notations
	\begin{align*}
		\eta_{i}(t)=& \left(\eta_{i,s}(t)\right)_{s \in \mathcal{S}},\\
		\bar{\eta}_{k}(t):=& \frac{1}{\left |V_k \right |}\sum_{i \in V_k}\eta_{i}(t).
	\end{align*}
	
	It is easy to see that $\E \left(\eta_{i}(t)\right)=z_{i}(t).$ Also, since $i \in V_k$ and $j \in V_l$, $i$ and $j$ are different for $k \neq l$, hence the corresponding processes are independent, so
	\begin{align*}
		z_{i,s'}(t)z_{j,r}(t)=&\E \left(\eta_{i,s'}(t)\eta_{j,k}(t)\right),\\
		\bar{z}_{k,s'}(t)\bar{z}_{l,r}(t)=& \E \left(\bar{\eta}_{k,s'}(t)\bar{\eta}_{l,r}(t)\right).
	\end{align*}
	Therefore, \eqref{eq:hard_work2} can be bounded from above by
	\begin{align}
		\label{eq:hard_work3}
		\E \left[\kappa\ \sum_{(k,l) \in \mathcal{I}} \left| \frac{1}{|V_k|}\sum_{i \in V_k}\sum_{j \in V_l} \frac{a_{ij}}{\bar{d}}\eta_{i,s'}(t)\eta_{j,r}(t)-\bar{w}_{kl}\bar{\eta}_{k,s'}(t)\bar{\eta}_{l,r}(t) \right| \right].	
	\end{align}
	
	The upper bound we aim to obtain does not depend on the artificial randomness just introduced, hence the expectation is ignored.
	
	We make some algebraic manipulation to end up with edge densities needed for Szemer\'edi's regularity lemma. We use the notation
	\begin{align*}
		V_{k,s}(t):=\left\{\left. i \in V_k \right| \eta_{i,s}(t)=1\right\}.
	\end{align*}
	
	Then
	\begin{align*}
		& \frac{1}{|V_k|}\sum_{i \in V_k}\sum_{j \in V_l} \frac{a_{ij}}{\bar{d}}\eta_{i,s'}(t)\eta_{j,r}(t)=\frac{1}{\left |V_k \right | \bar{d}}e \left(V_{k,s'}(t),V_{l,r}(t) \right)=\\
		& \frac{\left |V_l \right |}{ \bar{d}}\rho \left(V_{k,s'}(t),V_{l,r}(t) \right) \frac{\left| V_{k,s'}(t)\right|}{\left|V_k \right|}\frac{\left| V_{l,r}(t)\right|}{\left|V_l \right|}=\frac{\kappa}{p}\rho \left(V_{k,s'}(t),V_{l,r}(t) \right) \bar{\eta}_{k,s'}(t)\bar{\eta}_{l,k}(t).
	\end{align*}
	By recalling \eqref{eq:w_bar}, the inside of \eqref{eq:hard_work3} can be rewritten as
	\begin{align}
		\label{eq:hard_work4}
		\frac{\kappa^2}{p} \sum_{(k,l) \in \mathcal{I}} \left| \rho\left(V_{k,s'}(t),V_{l,r}(t)\right)-\rho\left(V_{k},V_{l}\right)  \right| \bar{\eta}_{k,s'}(t)\bar{\eta}_{l,r}(t).
	\end{align}
	Note that the summands of \eqref{eq:hard_work4} are $O(1)$. 
	
	Using Szemer\'edi's lemma to \eqref{eq:hard_work4} is relatively straightforward from now on. We still have to deal with non-$\varepsilon'$-regular $k,l$ pairs, and pairs where either $\left|V_{k,s'}(t) \right| \leq \varepsilon' \left|V_k \right|$ or $\left|V_{l,r}(t) \right| \leq \varepsilon' \left|V_l \right|$. The former set of pairs are denoted by $\mathcal{I}_1$ and the latter by $\mathcal{I}_2$, and $\mathcal{I}_3:=\mathcal{I}\setminus \left(\mathcal{I}_1 \cup \mathcal{I}_2\right)$ denotes the non-problematic pairs.
	
	Then from $\left|\mathcal{I}_1 \right| \leq \varepsilon' {K \choose 2} \leq \varepsilon' K^2$ we have
	\begin{align*}
		\frac{\kappa^2}{p} \sum_{(k,l) \in \mathcal{I}_1} \left| \rho\left(V_{k,s'}(t),V_{l,r}(t)\right)-\rho\left(V_{k},V_{l}\right)  \right| \bar{\eta}_{k,s'}(t)\bar{\eta}_{l,r}(t)=O \left(\varepsilon'\kappa^2 K^2 \right)=O\left(\varepsilon'\right).
	\end{align*}
	
	$(k,l) \in \mathcal{I}_2$ is equivalent with $\bar{\eta}_{k,s'}(t) \leq \varepsilon'$ or $\bar{\eta}_{l,k}(t) \leq \varepsilon'$, yielding
	\begin{align*}
		& \frac{\kappa^2}{p} \sum_{(k,l) \in \mathcal{I}_2} \left| \rho\left(V_{k,s'}(t),V_{l,r}(t)\right)-\rho\left(V_{k},V_{l}\right)  \right| \bar{\eta}_{k,s'}(t)\bar{\eta}_{l,r}(t)\leq \\
		& \frac{\varepsilon' \kappa^2}{p} \sum_{(k,l) \in \mathcal{I}_2} 1=O \left(\varepsilon' \kappa^2 K^2\right)=O\left(\varepsilon'\right).
	\end{align*}
	
	Finally, $(k,l) \in \mathcal{I}_{3}$ gives
	\begin{align*}
		& \left| \rho\left(V_{k,s'}(t),V_{l,r}(t)\right)-\rho\left(V_{k},V_{l}\right)  \right| < \varepsilon' \Rightarrow \\
		& \frac{\kappa^2}{p} \sum_{(k,l) \in \mathcal{I}_3} \left| \rho\left(V_{k,s'}(t),V_{l,r}(t)\right)-\rho\left(V_{k},V_{l}\right)  \right| \bar{\eta}_{k,s'}(t)\bar{\eta}_{l,r}(t)\leq \\
		& \frac{\varepsilon' \kappa^2}{p} \sum_{(k,l) \in \mathcal{I}_2} 1 =O \left(\varepsilon' \kappa^2 K^2\right)=O\left(\varepsilon'\right).
	\end{align*}

	This ensures that $\sup_{0 \leq t \leq T} \kappa \sum_{k=1}^{K} \left \|g_k(t) \right \|_1=O\left( \varepsilon' \right)$ indeed holds, concluding the proof of Theorem \ref{t:Szemeredi}.
	
	\qed
	
	\bibliographystyle{abbrv}
	\bibliography{mf}

\end{document}